\documentclass[12pt,reqno]{amsart}
\usepackage{amsmath,amsfonts,amsthm,amssymb,color,hyperref}

\usepackage[T1]{fontenc}
\usepackage{subfig} 
\usepackage{graphicx}

\makeatletter
     \def\section{\@startsection{section}{1}%
     \z@{.7\linespacing\@plus\linespacing}{.5\linespacing}%
     {\bfseries
     \centering
     }}
     \def\@secnumfont{\bfseries}
     \makeatother

\numberwithin{equation}{section}

\theoremstyle{definition}
\newtheorem{dfn}{Definition}[section]
\newtheorem{ass}{Assumption}[section]  
\theoremstyle{plain}
\newtheorem{thm}{Theorem}[section]
\newtheorem{pro}{Proposition}[section]
\newtheorem{cor}{Corollary}[section]
\newtheorem{lmm}{Lemma}[section]
\theoremstyle{definition}
\newtheorem{rem}{Remark}[section]
\newtheorem{exa}{Example}[section]

\newcommand{\R}{\mathbb{R}}
\newcommand{\E}{\mathbb{E}}

\newcommand{\la}{\langle}
\newcommand{\ra}{\rangle}
\renewcommand{\H}{\mathcal{H}}
\newcommand{\1}{\mathbf{1}}

\newcommand{\changes}[1]{\textcolor{black}{#1}}

\begin{document}
\title[Optimal prediction of functional data]{On optimal prediction of missing functional data with memory}

\author[Ilmonen]{Pauliina Ilmonen}
\address{Department of Mathematics and Systems Analysis, Aalto University School of Science, FINLAND} 
\email{pauliina.ilmonen@aalto.fi}
\author[Shafik]{Nourhan Shafik}
\address{Department of Mathematics and Statistics, University of Helsinki, FINLAND} 
\email{nourhan.shafik@helsinki.fi}
\author[Sottinen]{Tommi Sottinen}
\address{Department of Mathematics and Statistics, University of Vaasa, FINLAND}
\email{tommi.sottinen@iki.fi}
\author[Van Bever]{Germain Van Bever}
\address{European Center of Advanced REsearch in Economics and Statistics, Universit\'e libre de Bruxelles $\&$ Department of Mathematics,  University of Namur, BELGIUM}\email{germain.van.bever@ulb.be}
\author[Viitasaari]{Lauri Viitasaari}
\address{Department of Information and Service Management, Aalto University School of Business, FINLAND}
\email{lauri.viitasaari@aalto.fi}

\begin{abstract}
This paper considers the problem of reconstructing missing parts of functions based on their observed segments. 
It provides, for Gaussian processes and arbitrary bijective transformations thereof, theoretical expressions for the $L^2$-optimal reconstruction of the missing parts. 
These functions are obtained as solutions of explicit integral equations. 
In the discrete case, approximations of the solutions provide consistent expressions of all missing values of the processes. 
Rates of convergence of these approximations, under extra assumptions on the transformation function, are provided. 
In the case of Gaussian processes with a parametric covariance structure, the estimation can be conducted separately for each function, and yields nonlinear solutions in presence of memory.
Simulated examples show that the proposed reconstruction indeed fares better than the conventional interpolation methods in various situations. 
\end{abstract}

\maketitle



\maketitle

\section{Introduction}
\label{Sec:intro}
 
In functional data analysis, the observed units are random curves $(Y^1)_{t\in\mathcal{I}},\dots$, $(Y^J)_{t\in\mathcal{I}}$ defined on some domain $\mathcal{I}$. 
The standard setting, also adopted throughout this paper, is to assume that $t$ represents the time at which the functions are observed and that $\mathcal{I}=[0,T]\subset\mathbb{R}$, for $T>0$.
There is a vast body of work in functional data analysis, which often extends classical multivariate techniques to this particular setting. See, for example, \cite{Ramsay}, \cite{Fer2011}, \cite{Wanetal2016} and references therein.

While many procedures assume that the curves are fully observed, this will not be the case in most instances so that one has to proceed through a reconstruction step. 
The need for this step comes most often from the fact that only discretized measurements from each curve are available. 
More precisely, the collected data takes the form
$$Y_{ji}=Y^j_{t_{ji}},\ j=1,\dots,J,\ i=1,\dots,N_j,$$
where, in full generality, the measurement times $t_{ji}$ could vary in number or location within the curves themselves. 
Recovering the curves from their discrete measurements has been extensively explored in the statistical literature. 
Reconstruction methods typically depend on assumptions on the generating process of the curves $Y(t)$. They also, and most importantly, depend on the way the functions are discretized, that is, on the nature of the locations $t_{ji}$.

The most classical example is to assume that $\{t_{ji}\ :\ j=1,\dots,J,; i=1,\dots,N_j\}$ is increasingly dense in $[0,T]$ as $\min_jN_j$ gets larger. 
This will be the case, for example, for regularly observed data (i.e. $t_{ji}=iT/N$, $i=1,\dots,N$ for some $N$) or for dense randomly observed data (i.e. $t_{ji}\sim\mathcal{U}([0,T])$ are i.i.d. uniform random locations in $\mathcal{I}$, with $N_j$ large). In such situation, one typically proceeds with standard smoothing techniques, which include classical penalized regression on spline or Fourier basis functions (see \cite{FerVie2006} or \cite{UllFin2013}).
 In the more general setting of sparse functional data (for which $N_j$ is small and $t_{ji}$ are still i.i.d. uniform random variables), one can proceed with estimating the common mean and covariance functions of $(Y^j)_{t\in[0,T]}$ and reconstructing the functions under the normality assumption (see, for example, \cite{Jametal2000} and \cite{Yaoetal2005}). 

 Another situation which requires reconstructing the functional observation is the case of \emph{fragmented data}. 
Such data arises in situations where the curves $Y^j$ can not reasonably be observed regularly or randomly in the whole time domain $[0,T]$, but rather on a subinterval (or union thereof) $\Delta_j\subset[0,T]$. Each curve is observed, either discreetly or fully, on $\Delta_j$. Treated as fragments of a general function, the objective is then to reconstruct $Y^j$ on $[0,T]$. 
Proposals of curve reconstruction (\changes{in the setting described above}) include \cite{DelHal2016}, but also \cite{DelHal2013}, in the classification setting or \cite{Lie2013} and \cite{Goletal2014}, for prediction purposes. Similarly, reconstructing the covariance operator of $Y$ from fragmented data has been considered in \cite{DesPan2019} and \cite{Deletal2020}. \changes{Other proposals, building on basis expansions of the covariance operator include \cite{ZhangChen2022,Linetal2021}. Finally, \cite{LinWang2022} propose an example of application of such reconstruction for health data.}

There also exists several proposals for \emph{optimal} reconstruction of fragmented data. They include \cite{Kra2015}, \cite{LieRam2019} or, more recently, \cite{KneLie2020}. In the latter, the authors construct a linear operator $L$ which minimizes the mean squared error loss $\E[\{Y^j_s-[L((Y^j)_{t\in \Delta_j})]_s\}^2]$ at any $s\notin \Delta_j$. The estimation of $L$ is based on $\{(Y^j)_{t\in \Delta_j}: j=1,\dots,J\}$. Importantly, in all the references above, the reconstruction of $Y^j$ is only possible given the knowledge of the full dataset.


\changes{The approach adopted in this paper breaks from the previous literature in that it aims at providing optimal reconstruction of missing fragments $(Y^j)_{s\notin\Delta_j},\ j=1\dots,J$ in the context of stochastic processes with memory. The reconstruction does not need to be linear in the observed fragments $(Y^j)_{t\in \Delta_j}$. Interestingly, when the covariance structure is known or can be parametrically estimated, the reconstruction can be done on a curve-by-curve basis.}
 Throughout, we will assume that the observed set is a union of interval of the form
 \begin{equation}
\label{eq:partition-form}
\Delta = [t_L^1,t_U^1]\cup [t_L^2,t_U^2] \cup \ldots [t_L^n,t_U^n].
\end{equation}

 More precisely, the problem under consideration is to estimate the conditional expectation $\hat Y^j_s:=\E[Y_s^j|\mathcal{F}^j_\Delta]$, where $Y_s^j$ is a missing value of the function $Y^j$ and $\mathcal{F}_\Delta^j$ is the filtration generated by the observed parts of $Y^j$ itself. Note that $(\hat Y^j)_{s\in[0,T]}$ is indeed the optimal $L^2$ reconstruction of $Y^j$ in the sense that it minimizes 
 $$\textcolor{black}{\int_0^T\E[(Y^j_s-\hat Y^j_s)^2]ds.}$$

The contributions of this paper are threefold. In the first part, we provide an explicit expression for $\E[X_t|\mathcal{F}_\Delta]$, where $(X)_{t\in[0,T]}$ is a centered and separable Gaussian process that admits an integral representation. More precisely, under mild assumptions (discussed in Section 2), we make use of the Fredholm representation of $X$,
$$
X_t = \int_0^T K(t,s)dW_s,
$$
with $K(t,s)\in L^2([0,T]^2)$ a deterministic kernel and $W$ a Brownian motion. 
This in turn allows us to derive a Wiener integral representation of $\E[X_t|\mathcal{F}_\Delta]$, where the kernel is shown to satisfy certain integral equations. Importantly, in such setting, the curve reconstruction can be conducted on a curve-by-curve basis and presents a straightforward empirical expression.

In the second part, we provide explicit expressions for $\E[Y_s^j|\mathcal{F}^j_\Delta]$ under the assumption that $Y^j_t=f(t,X^j_t)$, $t\in[0,T]$, is some bijective transformation of a Gaussian process $(X^j)_{t\in[0,T]}$. Also in this situation, we show that, once the integral representation of $X^j$ is known, the regular conditional law of $Y^j$ can be computed by solving certain integral equations. 
Note that processes of the type $Y_t = f(t,X_t)$ form a large and flexible class of models. Indeed, it is known (see, e.g. \cite{viitasaari-ilmonen2020}) that the processes in this class can have arbitrary one-dimensional marginal distributions and approximate arbitrarily well any covariance structure. 
\changes{In comparison to \cite{KneLie2020}, the $L^2$-optimal reconstruction for such a general $f$ need not be linear in the observed values (see Theorem \ref{thm:prediction-FDA}). 
On the other hand, if $f$ is indeed linear, then the underlying process is Gaussian and it is well-known that, in this case, the optimal reconstruction is linear.}

Our third contribution is to provide a method to estimate both the bijective transformation and the solution to the integral equations, allowing us to estimate optimal $L^2$-predictors with minimal assumptions. 
\changes{Interestingly, if $f$ is known (as is the case, for example, in the first part of this paper) and if the covariance is \emph{simple}, then the reconstruction can be conducted separately for each functional observation provided. The covariance structure $R(t,s)$ is deemed \emph{simple} if it is either known or belongs to a parametric model (e.g. fractional Brownian motions) for which a single curve allows for estimation.}
This highlights the disentanglement between, on one side, the estimation of $f$ and $R$, which requires the whole dataset, and, on the other side, the reconstruction of $(X^j)_{t\notin \Delta_j}$ on the sole basis of $(X^j)_{t\in \Delta_j}$. 
We illustrate the estimation procedure for several Gaussian processes which include fractional Brownian motions. These processes have been studied intensively during the last decades, see, for example~\cite{Mishura-2008}.

The rest of the paper is organized as follows. 
In Section \ref{sec:Gaussian-prediction}, we provide optimal predictors for Gaussian processes. 
Throughout Section \ref{sec:Gaussian-prediction}, we assume that the underlying integral representation structure of the Gaussian process is known. In Section \ref{sec:FDA-prediction}, we drop the assumption of Gaussianity and known related integral structures. 
In particular, we explain how our approach can be used to approximate the optimal predictor by using the observed data directly and with posing minimal assumptions. 
\changes{We close Section \ref{sec:FDA-prediction} by providing convergence rates of the proposed estimators, under an additional regularity assumption on $f$.} 
We end the paper with a numerical study, provided in Section \ref{sec:simulations}, \changes{where we analyze both simulated and real data settings.}
The appendix to this paper is in two parts. In Appendix \ref{sec:appendix-Gaussian}, we recall some preliminaries on Gaussian analysis and provide some necessary results that guarantee the existence of solutions to our integral equations. Finally, all the proofs are contained in Appendix \ref{sec:proofs}.

\section{Bridge prediction laws for partially observed Gaussian processes}
\label{sec:Gaussian-prediction}

Let $X=(X_t)_{t\in [0,T]}$ be a centered, i.e. be such that $\E[X_t]=0$ for all $t\in [0,T]$, and separable Gaussian process on a probability space $(\Omega,\mathcal{F}(X),\mathbb{P})$. Furthermore, we assume throughout that the variance of $X$ is uniformly bounded, i.e.
\begin{equation}
\label{eq:bounded-var}
\sup_{t\in[0,T]}\E[X_t^2] < \infty.
\end{equation}
We stress that assuming uniformly bounded variance is a very mild condition. Most reasonable examples indeed fulfill~\eqref{eq:bounded-var}. As a particular example, continuity of $X$ (on the compact interval $[0,T]$, hence uniform continuity) is sufficient to guarantee both separability and \eqref{eq:bounded-var}.

If a centered and separable Gaussian process $X$ has an integrable variance function, i.e.
\begin{equation}
\label{eq:integrable-variance}
\int_0^T \E[X_t^2] dt < \infty,
\end{equation}
then there exists (for details, we refer to Appendix \ref{sec:appendix-Gaussian} and references therein) a standard Brownian motion $W$ and a deterministic kernel $K \in L^2\left([0,T]^2\right)$ such that 
\begin{equation}
\label{eq:fredholm}
X_t = \int_0^T K(t,s)dW_s
\end{equation}
holds in law. Note that~\eqref{eq:integrable-variance} trivially holds under~\eqref{eq:bounded-var}.
Rephrasing, under~\eqref{eq:bounded-var}, we assume that $X$ belongs to the separable Hilbert space $L^2([0,T])$ almost surely. 

\begin{rem}
The integral in \eqref{eq:fredholm} stands for the Wiener integral with respect to the Brownian motion. In general, Wiener integrals 
$$
\int_0^T f(s) dW_s
$$
can be defined for any function $f \in L^2([0,T])$. A more complete exposure on the topic is available in Appendix \ref{sec:appendix-Gaussian}.
\end{rem}
For $t\in[0,T]$, let $H^X_t$ be the first chaos of $X$, that is the closure (in $L^2(\Omega)$) of the linear space spanned by $X_s, {s\in [0,t]}$ (see also Definition \ref{def:first-chaos}).
We introduce the following assumption that stands throughout the article.
\begin{ass}
\label{assu:standing}
The centered and separable Gaussian process $X$ satisfies \eqref{eq:bounded-var}. Furthermore, there exist a kernel $K$ and a Brownian motion $W$ such that \eqref{eq:fredholm} holds exactly, and  $H^X_T = H^W_T$.
\end{ass}
Assumption \ref{assu:standing} is a natural assumption and certainly not very restrictive. 
Indeed, whenever \eqref{eq:bounded-var} is satisfied, we obtain that \eqref{eq:fredholm} holds in law. 
Assumption \ref{assu:standing} therefore merely states that the underlying Brownian motion drives the process $X$ directly and \eqref{eq:fredholm} holds as an equality instead of only as a representation in law. 
The additional condition $H^X_T = H^W_T$ is also natural and not restrictive. 
\changes{Indeed, this simply means that, when the whole interval $[0,T]$ is considered, the process $X$ shares the same amount of randomness as the driving Brownian motion $W$. 
More precisely, write the Kosambi-Karhunen-Lo\`eve expansion of  $W$ as
$$
W_t = \sum_{k=1}^\infty \sqrt{\lambda_k} \psi_k(t)\xi_k, \textrm{ with }\xi_k=\int_0^T\frac{1}{\sqrt{\lambda_k}}W_t\psi_k(t)dt
$$
and where $\xi_k \sim N(0,1)$, with $\lambda_1\geq \lambda_2\geq\cdots>0$ and $\{\psi_k:k\geq 1\}$ is a deterministic orthonormal basis of $L^2(0,T)$ which diagonalises the covariance operator of $W_t$. 
A necessary and sufficient condition for $H^X_T = H^W_T$ to hold is for $X_t$ to be of the form
$$
X_t = \sum_{k=1}^\infty c_k\phi_k(t)\xi_k,
$$
where the random variables $\xi_k$ are taken from the expansion of $W$ and $\{\phi_k:k\geq 1\}$ is another orthonormal basis of $L^2(0,T)$ and $c_k$ is a deterministic sequence of strictly positive constants.
Note that, in the latter expansion, the condition $c_k>0$ $\forall k$ imposes that $X$ is not finite dimensional. 
On the other hand, note also that the path properties of $X$ are not related to the coefficients $\xi_k$ but rather to the functions $\phi_k$ (the eigenfunctions of the covariance operator of $X$). 
Thus, Assumption \ref{assu:standing} is satisfied also for many processes $X$, e.g. with smooth sample paths. 
An example of a Gaussian process $X$ satisfying Assumption \ref{assu:standing} with smooth sample paths is the $n$-folded fractional Brownian motion, studied e.g. in \cite{Sottinen-Viitasaari2018}.
Other explicit examples of processes satisfying $H_T^X = H_T^W$ but with different sample path regularity are provided in Section \ref{sec:simulations}. 
We also remark that, in many interesting cases, the stronger relation $H^X_t=H^W_t$ holds for all $t\in[0,T]$. This is the case in all the examples considered in Section \ref{sec:simulations} and e.g. in the case of the $n$-folded fractional Brownian motion. 
Finally, we also point out that the condition $H_T^X=H_T^W$ is dropped in the discrete approach detailed later in Theorem \ref{thm:discrete}.}

An interesting and widely studied subclass of processes satisfying Assumption~\ref{assu:standing} are called Volterra Gaussian processes. They admit a representation
\begin{equation}
\label{eq:volterra}
X_t = \int_0^t K(t,s)dW_s.
\end{equation}
The expression \eqref{eq:fredholm} is called the ``Fredholm representation of $X$'', stemming from the fact that $X_t$ can only be computed with the knowledge of the kernel $K(t,s)$ for all $s\in[0,T]$. Similarly, \eqref{eq:volterra} is coined the ``Volterra representation of $X$'' and its kernel $K(t,s)$ is of Volterra-type, that is, is a Fredholm kernel such that $K(t,s)=0$ for $s>t$. 
\begin{exa}
Let $X$ be an $H$-self-similar Gaussian process, i.e. $X_{at}\stackrel{\mathcal{L}}{=}a^H X_t$ for all $a\geq 0$ and all $t\geq 0$. Then $H^X_t = H^W_t$ for all $t\in[0,T]$ if and only if $X$ has a Volterra representation \eqref{eq:volterra}. The latter holds whenever $X$ is purely non-deterministic. That is, when $\bigcap_{t>0} H_t^X$ is trivial, i.e. spanned only by constants. For details and for the proof of these facts, we refer to \cite{Yazigi-2015}.

\end{exa}
We are interested in the prediction law $X|\mathcal{F}_{\Delta}$, where 
\begin{equation}
\label{eq:partition-form2}
\Delta = [t_L^1,t_U^1]\cup [t_L^2,t_U^2] \cup \ldots \cup [t_L^n,t_U^n]
\end{equation}
for some $0\leq t_L^1<t_U^1<t_L^2<t_U^2 <\ldots t_L^n < t_U^n \leq T$ and 
$\mathcal{F}_{\Delta}$ is the $\sigma$-algebra generated by $X$ on $\Delta$, i.e.
\begin{equation*}
\mathcal{F}_{\Delta} = \sigma(X_t : t \in \Delta).
\end{equation*}
Setting $t_U^0=0$ and $t_L^{n+1}=T$, we aim at reconstructing the missing values on each missing subinterval $(t_U^k,t_L^{k+1}),k=0,\ldots,n$. From a practical point of view, this means that we observe $X_s$ for $s\in \Delta$ only, and our aim is to predict the missing values at $s \in [0,T]\setminus\Delta$. 

The following result provides the best predictor in the $L^2(\Omega)$-sense. 
\changes{Note that, in this particular Gaussian setting, the best predictor is a linear one. 
As such, the following theorem corresponds to the results of \cite{KneLie2020} in the Gaussian case, though the formulation adopted here is via Fredholm representation and through its kernel $K$. 
Theorem~\ref{thm:Gaussian-fredholm-prediction} (together with the other results in this section) paves the way for the more general case, presented in Section \ref{sec:FDA-prediction}, in which linearity of the optimal reconstruction does not hold true.}
\begin{thm}
\label{thm:Gaussian-fredholm-prediction}
Let $X$ be a stochastic process satisfying Assumption \ref{assu:standing} and let $\mathcal{F}_\Delta=\sigma(X_t:t\in\Delta)$. Then for each $t\in [0,T]\setminus \Delta$ we have
\begin{equation}
\label{eq:prediction-fredholm}
\hat{X}_t = \E\left[X_t | \mathcal{F}_\Delta\right] = \int_0^T f_t(s)dW_s, 
\end{equation}
where $f_t \in L^2 ([0,T])$ is the solution of the (system of) Fredholm integral equations
\begin{equation}
\label{eq:fredholm-system}
\int_0^T f_t(s)K(u_k,s)ds = \int_0^T K(t,s)K(u_k,s)ds, 
\end{equation}
for $k=1,\dots,n$ and for all $t_L^k\leq u_k<t_U^k$, that satisfies
$$\int_0^T f_t(s)g_\xi(s)ds = 0\textrm{, for all } g_\xi\in L^2([0,T]) \textrm{ such that } \int_0^T g_\xi(s)K(u_k,s)ds = 0.$$
\end{thm}
\begin{rem}
\label{rem:uniqueness}
In general, integral equations of the first kind such as (\ref{eq:fredholm-system}) are ill-posed problems and uniqueness of the solution is not guaranteed. 
Indeed, from $H_T^X = H_T^W$ it follows that each $\xi \in H_T^X$ can be represented as $\xi = \int_0^T g_\xi(s)dW_s$ for some $g_\xi\in L^2([0,T])$. Consequently, we may choose any $\xi = \int_0^T g_\xi(s)dW_s \in H_T^X$ orthogonal to $\mathcal{F}_\Delta$, and, since then $\int_0^T g_\xi(s)K(u_k,s)ds = 0$ for $t_L^k\leq u_k<t_U^k$, we observe that $f_t + g_\xi$ solves \eqref{eq:fredholm-system} whenever $f_t$ does. As $g_\xi$ can, in the general case, be chosen rather arbitrarily, the solution to \eqref{eq:fredholm-system} is clearly not unique. 
On the other hand, there is exactly one solution $f_t$ of \eqref{eq:fredholm-system} such that \eqref{eq:prediction-fredholm} holds (see the proof of Theorem~\ref{thm:Gaussian-fredholm-prediction} in Appendix~\ref{sec:proofs}). Indeed, existence and uniqueness of the conditional expectation allows to fix a unique representative within the set of solutions of \eqref{eq:fredholm-system}.
\end{rem}
\begin{rem}
The conditional mean in \eqref{eq:prediction-fredholm} is expressed as a Wiener integral with respect to the driving Brownian motion $W$. 
Such an expression is useful when one is interested in theoretical properties of $\hat{X}$ or for simulation purposes. 
In practice, however, the driving Brownian motion is not observable on $[0,T]$ and thus one needs to transform \eqref{eq:prediction-fredholm} into something computable from the observations $X_s,s\in \Delta$, directly. The approximation of $\hat{X}_t$ in terms of the process $X$ itself is discussed in Section \ref{sec:FDA-prediction}.
\end{rem}
The result below provides the regular law of $X$, conditional to $\mathcal{F}_\Delta$.  
It will prove useful when studying non-Gaussian processes in Section \ref{sec:FDA-prediction}.
\begin{thm}
\label{thm:Gaussian-regular-law}
Let $X$ be a stochastic process satisfying Assumption \ref{assu:standing} and let $\mathcal{F}_\Delta=\sigma(X_t:t\in\Delta)$. Then, the regular conditional law $X|\mathcal{F}_\Delta$ is Gaussian with random mean given by \eqref{eq:prediction-fredholm} and deterministic covariance given, for $t,s\in [0,T]\setminus \Delta$, by
\begin{equation}
\label{eq:rho}
\rho(t,s|\Delta) = \int_0^T \left[K(t,u)-f_t(u)\right]\left[K(s,u)-f_s(u)\right]du,
\end{equation}
where $f$ is the unique solution of~\eqref{eq:fredholm-system} in Theorem~\ref{thm:Gaussian-fredholm-prediction}.
\end{thm}

The next theorem particularizes Theorem~\ref{thm:Gaussian-fredholm-prediction} to the case of Volterra Gaussian processes.

\begin{thm}
\label{thm:Gaussian-prediction}
Let $X$ be a Volterra process satisfying Assumption \ref{assu:standing} with representation \eqref{eq:volterra} and let $\mathcal{F}_\Delta=\sigma(X_t:t\in\Delta)$.
Then, for each $t\in [0,T]\setminus \Delta$ we have
\begin{equation}
\label{eq:prediction}
\hat{X}_t = \E\left[X_t | \mathcal{F}_\Delta\right] = \int_0^T f_t(s)dW_s, 
\end{equation}
where, for each $k=1,2,\ldots,n$ and $t_L^k\leq s\leq t_U^k$, the function $f_t \in L^2([0,T])$ is a solution of the recursive system of Volterra integral equations of the first kind given by, for $t_L^k\leq u_k\leq t_U^k$,
\begin{equation}
\label{eq:thm-volterra-eq}
\int_{t_L^k}^{u_k} [f_t(s)-K(t,s)]K(u_k,s)ds = \int_0^{t_L^k} [K(t,s)-f_t(s)]K(u_k,s)ds
\end{equation}
with the boundary conditions
\begin{equation}
\label{eq:thm-boundary-eq}
\int_{t_U^{k-1}}^{t_L^k}f_t(s) K(t_L^k,s)ds = \int_0^{t_L^k} K(t,s)K(t_L^k,s)ds - \int_0^{t_U^{k-1}} f_t(s)K(t_L^k,s)ds.
\end{equation}
\end{thm}
\begin{rem}
As in the case of Theorem \ref{thm:Gaussian-fredholm-prediction}, the solution of \eqref{eq:thm-volterra-eq}-\eqref{eq:thm-boundary-eq} is not unique. However, only one amongst them satisfies \eqref{eq:prediction}. It can be determined imposing the same condition as in Theorem~\ref{thm:Gaussian-fredholm-prediction}.
\end{rem}

The following corollary covers the case $n=2$, $t_L^1=0$, and $t_U^2=T$, i.e. the case where $\Delta=[0,t_U^1]\cup[t_L^2,T]$ is such that only one subinterval is missing in $[0,T]$. In that situation, it provides the unique kernel $f_t$ satisfying \eqref{eq:prediction} of Theorem~\ref{thm:Gaussian-prediction} under the hypothesis that $H^X_t = H^W_t$ for all $t\in[0,T]$. 

\begin{cor}
\label{cor:one-missing}
Let $X$ be a stochastic process satisfying Assumption \ref{assu:standing} and let $\mathcal{F}_\Delta=\sigma(X_t:t\in\Delta)$, with $\Delta = [0,t_U]\cup [t_L,T]$.
Suppose further that $H^X_t = H^W_t$ for all $t\in[0,T]$. Then for all $t\in[t_U,t_L]$ we have
\begin{equation}
\label{eq:one-missing}
\hat{X}_t = \E[X_t | \mathcal{F}_\Delta] = \int_0^{t_U} K(t,s)dW_s +c(t)\int_{t_U}^{t_L}K(t_L,s)dW_s+ \int_{t_L}^T g_t(s)dW_s,
\end{equation}
where $\textcolor{black}{c(t),t\in[t_U,t_L]}$ is given by
\begin{equation}
\label{eq:c}
c(t) = \frac{\int_{t_U}^{t}K(t,s)K(t_L,s)ds}{\int_{t_U}^{t_L}K^2(t_L,s)ds}
\end{equation}
and $g_t(s)$ is the solution to the Volterra integral equation of the first kind given by
\begin{equation}
\label{eq:int-eq}
\int_{t_L}^u g_t(s)K(u,s)ds = \int_{t_U}^{t_L}\left[K(t,s)-c(t)K(t_L,s)\right]K(u,s)ds\textcolor{black}{,\quad u\in[t_L,T]}.
\end{equation}
\end{cor}

The expression~\eqref{eq:one-missing} decomposes the dependence of $\hat X_t$ on the process $X$ in three sub-intervals. Note, however, that the assumption \textcolor{black}{$H^X_s = H^W_s$ for all $s\in[0,T]$}, guarantees
$$
\int_0^{t_U} K(t,s)dW_s +c(t)\int_{t_U}^{t_L}K(t_L,s)dW_s = \int_0^{t_U} K(t,s)-c(t)K(t_L,s)dW_s + c(t)X_{t_L},
$$
so that \eqref{eq:one-missing} rewrites
$$
\hat{X}_t = \E[X_t | \mathcal{F}_\Delta] = \Big(\int_0^{t_U} K(t,s)-c(t)K(t_L,s)dW_s\Big) + \Big(c(t)X_{t_L} + \int_{t_L}^T g_t(s)dW_s\Big).
$$
In particular, $\int_0^{t_U} K(t,s)-c(t)K(t_L,s)dW_s \in H_{t_U}^X$ is measurable as well. This removes any explicit integration within the missing interval $[t_U,t_L]$. 
Finally, note that computing $\hat X_t$ remains difficult as, indeed, the quantity $\int_{t_L}^T g_t(s)dW_s$ can not easily be expressed as a function of the observable process $X$. We address this remaining difficulty in the empirical case using a discretization argument; see Section~\ref{sec:FDA-prediction}.


The process $\hat X_t$ detailed in Corollary~\ref{cor:one-missing} is a Gaussian bridge, since $\hat X_{t_U}=X_{t_U}$ and $\hat X_{t_L}=X_{t_L}$. 
Indeed, this can be seen by setting $t=t_U$ in \eqref{eq:one-missing}, which in turn gives $c(t_U)=0$ and $g_{t_U}$, a solution to 
$
\int_{t_L}^u g_{t_U}(s)K(u,s)ds = 0,
$
leading to $g_{t_U}\equiv 0$. This shows 
$$\hat X_{t_U}=\int_0^{t_U}K(t_U,s)dW_s = X_{t_U}.$$
Similarly, for $t=t_L$, we again have
$
\int_{t_L}^u g_{t_L}(s)K(u,s)ds = 0,
$
giving $g_{t_L}\equiv 0$. Also, $c(t_L)=1$ and the first two terms give
$$
\hat X_{t_L}=\int_0^{t_U} K(t_L,s)dW_s + \int_{t_U}^{t_L}K(t_L,s)dW_s  = X_{t_L}.
$$

\begin{rem}
Note that, in the general situation of Theorem~\ref{thm:Gaussian-fredholm-prediction}, the conditional expectation $\hat X_t=\mathbb{E}[X_t|\mathcal{F}_\Delta]$ is a (generalized) Gaussian bridge (see~\cite{Sottinen-Yazigi-2014}). 
Note also that, even in the setting of Theorem~\ref{thm:Gaussian-prediction} with $n\geq 3$, showing that $\hat X_{t_U^i}=X_{t_U^i}$, $i=1,\dots,n$ and $\hat X_{t_L^j}=X_{t_L^j}$, $j=1,\dots,n$ is not as straightforward (compared to the case $n=2$, that is, under Corollary~\ref{cor:one-missing}). 
The same argument as above holds for $t = t_L^n$ and $t = t_U^1$, though. This is not the case for $t = t_U^{j}$, $j=2,\dots,n$ (or $t=t_L^i$, $i=1,\dots,n-1$), which cannot be covered by simply plugging into \eqref{eq:thm-volterra-eq}. 
It can be showed that the only solution to \eqref{eq:thm-volterra-eq} and \eqref{eq:thm-boundary-eq} satisfying the unicity condition satisfies $f_{t_U^k}(s) = K(t_U^k,s)$ and $f_{t_L^k}(s) = K(t_L^k,s)$, for all $k=1,\dots, n$. This gives $\hat{X}_{t_U^k} = X_{t_U^k}$ and $\hat{X}_{t_L^k} = X_{t_L^k}$ as expected.

%
\end{rem}


The following corollary provides the future prediction laws by setting $n=1$ and $t_L^1=0$.
\begin{cor}
\label{cor:future-prediction}
Let $X$ be a stochastic process satisfying Assumption \ref{assu:standing} and let $\mathcal{F}_\Delta=\sigma(X_t:t\in\Delta)$, with $\Delta = [0,t_U]$. Suppose further that $H^X_t = H^W_t$ for all $t\in[0,T]$. Then for all $t>t_U$ we have
\begin{equation}
\label{eq:prediction-future}
\hat{X}_t = \E[X_t | \mathcal{F}_\Delta] = \int_0^{t_U} K(t,s)dW_s.
\end{equation}
\end{cor}

\begin{rem}
Corollary \ref{cor:future-prediction} can be viewed as a generalization of the results provided in \cite{Sottinen-Viitasaari-2017b}, where the authors studied the future prediction law of fractional Brownian motions. 
Contrary to the latter, we here rely on functional analytic argument of invertibility of bounded linear operators. We also note that, without the additional assumption $H^X_t = H^W_t$, it is not clear whether \eqref{eq:prediction-future} is measurable with respect to $\mathcal{F}_\Delta$. 
\end{rem}

\begin{exa}
\label{ex:Bm}
Let $W$ be a Brownian motion, i.e. $K(t,s) = \1_{s\leq t}$. By direct computations, for $t\in(t_U^{j-1},t_L^{j})$, the solution is given by
$$
f_t(s) = \textbf{1}_{s\leq t_U^{j-1}} + \frac{t-t_U^{j-1}}{t_L^j - t_U^{j-1}}\textbf{1}_{t_U^{j-1}<s<t_L^j}.
$$ 
Thus
$$
\hat{W}_t = W_{t_U^{j-1}} + \frac{t-t_U^{j-1}}{t_L^j - t_U^{j-1}}(W_{t_L^j}-W_{t_U^{j-1}}),
$$
i.e. the optimal predictor is linear within each missing interval.
\end{exa}

Finally, we consider the approximation of~\eqref{eq:prediction-future} in practice, where $X$ is only observed on the discrete set
$$
\Delta_N = \{t_j : t_j \in \Delta,j=1,2,\ldots,N\}\subset\Delta.
$$
Let $X_N\in\mathbb{R}^N$ be the observed vector with components $(\textbf{X}_{N})_j = X_{t_j}$, $t_j \in \Delta_N$.
Let $\textbf{R}_{N} \in \mathbb{R}^{N\times N}$ denote the covariance matrix of $\textbf{X}_{N}$ and let $\textbf{b}_{N}(t) \in \mathbb{R}^N$, $t\in[0,T]$ denote the vector consisting of covariances $\E[X_t X_{t_j}],t_j \in \Delta_N$. Using representation \eqref{eq:fredholm}, we thus have
$$
(\textbf{R}_{N})_{ij} = \int_0^T K(t_i,s)K(t_j,s)ds
$$
and 
$$
(\textbf{b}_{N}(t))_j = \int_0^T K(t,s)K(t_j,s)ds.
$$
\begin{thm}
\label{thm:discrete}
Suppose $X$ satisfies Assumption \ref{assu:standing} and that $\textbf{R}_{N}$ has full rank. Then for any $t\in [0,T]\setminus \Delta_N$ we have
\begin{equation}
\label{eq:discrete-predictor}
\hat{X}_{t,N} = \E\left[X_t | \mathcal{F}_{\Delta_N}\right] = \sum_{t_j \in \Delta_N} a_j(t) X_{t_j},
\end{equation}
where $\textbf{a}(t) = (a_1(t),\ldots,a_N(t))^T$ is given by
\begin{equation}
\label{eq:vector-a}
\textbf{a}(t) = \textbf{R}_{N}^{-1}\textbf{b}_{N}(t).
\end{equation}
Moreover, if the sequence $\mathcal{F}_{\Delta_N}, N\geq 1$ increases to $\mathcal{F}_\Delta$,  then 
\begin{equation}
\label{eq:discrete-predictor-convergence}
\hat{X}_{t,N} \to \E\left[X_t | \mathcal{F}_{\Delta}\right] 
\end{equation}
in $L^2(\Omega)$. 
\end{thm}
\textcolor{black}{Note that the Gaussian prediction formulas \eqref{eq:discrete-predictor}-\eqref{eq:vector-a} in the discrete setting are well-known (see, for example, \cite{BerlThom2004}). For the sake of completeness, we however present their proofs. Even so,} it is worth emphasizing that Theorem \ref{thm:discrete} provides a ready-to-use practical approach for prediction based on the covariances $\textbf{R}_N$ and $\textbf{b}_N$ only. 
Moreover, the assumption that $\mathcal{F}_{\Delta_N}$ is increasing is also natural, as this includes, e.g., the case where $\Delta_N\subset\Delta_{N+1}$ and, at each step, the new sample point is chosen uniformly from $\Delta$ and $N\rightarrow\infty$. 
Note that, in general, it is a challenging task to find the correct solution to the integral equations in Theorem \ref{thm:Gaussian-fredholm-prediction} or Theorem \ref{thm:Gaussian-prediction}. 
The convergence provided in Theorem~\ref{thm:discrete} avoids this difficulty. \changes{Note that, on the other hand, discretization in Theorem \ref{thm:discrete} corresponds to the discrete approximation to the integral equations, and hence the convergence depends heavily on the chosen grid $\Delta_N$. This is also visible in Proposition \ref{prop:convergence-rate} which provides the rate of convergence.}

\section{Partially Observed Functional Data}
\label{sec:FDA-prediction}

In this section, we break from the Gaussian case and consider a general family of processes $(Y_t)_{t\in[0,T]}$ such that the marginal distribution of $Y_t$ may be arbitrary and such that $(Y_t)_{t\in[0,T]}$ display any (approximation of a) covariance structure. 
This can be achieved by setting $Y_t =f(t,X_t)$, for $X$ a Gaussian process satisfying Assumption \ref{assu:standing}. Discussion on these processes can be found in \cite{viitasaari-ilmonen2020}. \textcolor{black}{In the sequel, we denote by $R(t,s)$ the covariance function of the underlying Gaussian process $X$.}

The following result provides an explicit expression of the optimal $L^2(\Omega)$-reconstruction in this general setting.

%
%
%
\begin{thm}
\label{thm:prediction-FDA}
Let $Y_t =f(t,X_t)$, where $X$ is a Gaussian process satisfying Assumption \ref{assu:standing} \textcolor{black}{and the measurable mapping $f : [0,T]\times \mathbb{R} \mapsto \mathbb{R}$ is such} that $x\mapsto f(t,x)$ is bijective for all $t\in[0,T]$. Then for any $t\in[0,T]\setminus \Delta$ we have
\begin{equation}
\label{eq:non-gaussian-prediction}
\mathbb{E}\left[Y_t | \mathcal{F}_\Delta^Y\right] = \int_{-\infty}^\infty f\big(t,\hat{X}_t +z\sqrt{\rho(t,t|\Delta)}\big)\phi(z)dz,
\end{equation}
where $\mathcal{F}_\Delta^Y=\sigma(Y_t:t\in\Delta)$, $\hat X_t$ is provided in Theorem~\ref{thm:Gaussian-fredholm-prediction}, $\rho(t,t|\Delta)$ is given by \eqref{eq:rho} and $\phi$ denotes the standard normal density. 
\end{thm}

\begin{rem}\label{rem:role-of-t}
\changes{One would expect smaller estimation error when $t$ is close to the observed part $\Delta$. This is indeed visible in \eqref{eq:non-gaussian-prediction} where the value of $\rho(t,t|\Delta)$ becomes smaller as the distance of $t$ to $\Delta$ decreases. The integral then concentrates around the value $f(t,\hat{X}_t)$ as expected.} 
\end{rem}

Theorem~\ref{thm:prediction-FDA} puts in perspective the role played by $f$ in the reconstruction of the missing parts of $Y$. Should $f$ be known, then one could simply recover $X_t=(f(t,\cdot))^{-1}(Y_t)$ and apply Theorem~\ref{thm:Gaussian-fredholm-prediction}. Therefore, in this situation, it is not required to sample several processes to be able to reconstruct a single one. This reconstruction can be conducted on a curve-by-curve basis as in Section~\ref{sec:Gaussian-prediction}.

On the other hand, if $f$ is not known, as is arguably the case in most instances, then one should proceed with its estimation. 
\changes{The rest of this section considers the general situation of a sample of functional observations with missing intervals. 
Throughout, it is enough to assume that $f$ is injective since only its inverse is of interest. Without loss of generality, we can further assume that it is surjective on its image set. 
}

\subsection{Consistent estimation of $\mathbb{E}\left[Y_t | \mathcal{F}_\Delta^Y\right]$}

Let $(Y^j_t)_{t\in\Delta_j},j=1,2,\ldots, J$, be $J$ independent observations where, for each $j$, the observed part $\Delta_j$ is of the form \eqref{eq:partition-form}. Heuristically, if we have sufficiently many observations $Y^j_t$ for each $t\in[0,T]$, then we can estimate the function $f(t,X_t)$ and, by invertibility of $f$, recover the \emph{proxy} values for $X^j_t$. 
From these values, we can then estimate the covariance of the underlying process $X$. 
After that, we can estimate the prediction $\hat{X}_t$ and $\rho(t,t|\Delta)$, from which we get the prediction $\hat{Y}_t$ by replacing $f$, $\hat{X}_t$, and $\rho(t,t|\Delta)$ by their estimated values in \eqref{eq:non-gaussian-prediction}. 
In order to make this heuristic argument precise, we start by posing the following assumptions.
For our purposes, we state the following assumption on $f$.
\begin{ass}
\label{ass:f-continuous}
The process $(Y)_{t\in[0,T]}$ is of the form
\begin{equation}
\label{eq:model}
Y_t = f(t,X_t),
\end{equation}
where, for each $t\in[0,T]$, the transformation $z\mapsto f(t,z)$ is continuous and \textcolor{black}{increasing}. 
\end{ass}
\begin{rem}
\changes{The assumption that $z\mapsto f(t,z)$ is increasing is made for the sake of simplicity. See, e.g. Proposition \ref{prop:f-rep} below. By examining its proof, it is clear that it suffices to assume monotonicity and all results are valid for decreasing functions as well.}
\end{rem}

Note that, the continuity of $f$ in Assumption~\ref{ass:f-continuous} is equivalent to assuming that the distribution of $Y_t$ has full support \textcolor{black}{(on the image set of $f$)}, provided that the distribution of $X_t$ is continuous and has full support. Respectively, the bijectivity of $f$ amounts to assuming that the distribution of $Y_t$ is continuous under the same hypothesis on $X_t$. Note that the latter trivially holds in the context of the Gaussian processes $X$ used in this section. 
Remark also that no particular behavior in $t$ is assumed. In particular, the classical functional model $Y_t = h(t) + X_t$ (where each observation is a noisy version of a common mean $h$) always satisfies  Assumption~\ref{ass:f-continuous}, regardless whether $h$ is continuous or not. Finally, note that, under Assumption~\ref{ass:f-continuous}, the continuity of $f$ allows to approximate~\eqref{eq:non-gaussian-prediction} using standard Riemann-Stieltjes integration.


In general, the model $Y_t = f(t,X_t)$ is not identifiable. Indeed, for any $c_t>0$, $t\in[0,T]$  it holds,
$$
g(t,Z_t) = f(t,X_t),
$$ 
where $g(t,z) = f\left(t,c_tz\right)$ and $Z_t = X_t/c_t$ is a Gaussian process. To ensure identifiability, one then needs to impose standardization constraints on $X_t$. We impose the following assumption.
\begin{ass}
\label{ass:var-X-known}
The variance function $t\mapsto \E X_t^2$ of $X$ is known.
\end{ass}

\changes{Assumption \ref{ass:var-X-known} is always satisfied (assuming some randomness in the model), as by the above considerations we can always choose the variance function $\E X_t^2$ arbitrarily.}
For example, imposing ${\rm Var}(X_t)=1$ for all $t\in[0,T]$ amounts to choosing $c_t=\sqrt{\mathbb{E}[X_t]}$  above. 
\changes{Note that, with different chosen variance function (i.e. choosing different $c_t$ above), one obtains a different function $f$ which is then estimated from the data.}
Note also that any other similar condition ensuring the identifiability of the model would suffice. As another example, one could impose the restriction $\text{Var}(Y_t) = \E[X_t^2]$. 
In this situation as the variance of $Y_t$ can be estimated from the observations, we could recover the variance function of $X$ as well. 
Under this assumption, the transformation $f$ can be expressed in a way suitable for future estimation. Since $x\mapsto f(t,x)$ is bijective, the proof follows easily and is hence omitted. 

Note that, if, for some $t$, $\mathbb{E}[X_t^2]=0$, then $\textrm{Var}(Y_t)=0$, so that $Y_t=c$ with probability one. 
One can then trivially predict $\mathbb{E}\left[Y_t | \mathcal{F}_\Delta^Y\right]=c$. 
In the following, we only consider predictions for values of $t$ such that $\mathbb{E}[X_t^2]>0$.

\begin{pro}
\label{prop:f-rep}
Let $t\in[0,T]$ be fixed and let $F_t$ and $Q_t$ denote the cumulative distribution function and quantile function of the non constant random variable $Y_t$, respectively. Then, under Assumption \ref{ass:f-continuous}, we have
$$
f(t,z) = Q_t\left(\Phi_t(z)\right)
$$
and $X_t=(f(t,\cdot))^{-1}(Y_t)$, with
\changes{$$
(f(t,\cdot))^{-1}(z) = \Phi^{-1}_t\left(F_t(z)\right),
$$}
and where $\Phi_t$ is the cumulative distribution function of $N(0,\E X_t^2)$.
\end{pro}
\textcolor{black}{The proof of Proposition \ref{prop:f-rep} is straightforward, and hence omitted.}

We are now ready to introduce our estimation procedure for the predictor $\mathbb{E}\big[Y^j_t | \mathcal{F}_{\Delta_j}^Y\big]$, $t\in [0,T]\setminus \Delta_j$. 
We also assume that we are given the discrete net $\pi_N \subset [0,T]$ of evaluation points and the corresponding observed values $Y_{t_j},t_j \in \Delta_j\cap\pi_N$. 
For $t\in \pi_N$, let $K(t)=\{j:t\in\Delta_j\cap\pi_N\}$ denote the indices of observations $Y_j$ defined at $t$. 
Similarly, for $(t,s)\in\pi_N^2$, let $J(t,s)=\{j:t,s\in\Delta_j\cap\pi_N\}$ denote the indices of observations $Y_j$ defined at both $t$ and $s\in\Delta_j$. 
With a slight abuse of notation, we omit the dependence in $t$ or $(t,s)$ and simply use $K$ and $J$ for all $t,s \in \pi_N$. 
The estimation is done by using the following steps.\vspace{0.3cm}
\begin{enumerate}
\item \textbf{Estimation of $f$ and construction of a proxy of the process $X$}:\vspace{0.3cm}\\
For $t \in \pi_N$, define the empirical distribution function based on the observations $Y_t^k,k\in K$ by
$$
\hat{F}_{t,K}(z) = \frac{1}{|K|+1} \sum_{k\in K} \1_{Y^k_t \leq z}.
$$
Note that here the summation is over all $|K|:=|K(t)|$ observed values $Y^k_t$, and we divide with $|K|+1$ instead of the usual $|K|$ to ensure that $\hat{F}_{t,K}(Y_t^k) \in (0,1)$ for any observed value $Y_t^k$.
Let $\hat{Q}_{t,K}$ denote the empirical quantile function  of $Y_t$ and define the estimator $\hat{f}$ of $f$ by setting, for each $t\in \pi_N$ and $z\in \R$, 
\begin{equation}
\label{eq:f-estimator}
\hat{f}_K(t,z) = \hat{Q}_{t,K}\left(\Phi_t(z)\right).
\end{equation}
The proxy $X_{t,K}^{\text{prox}}$ for the value $X_t$ is defined by
\begin{equation}
\label{eq:proxy-def}
X_{t,K}^{\text{prox}} = \Phi^{-1}_t\left(\hat{F}_{t,K}(Y_t)\right).
\end{equation}
Since $\hat{F}_{t,K}(Y_t) \in (0,1)$, $X_{t,K}^{\text{prox}} \in (-\infty,\infty)$ is well-defined.
\vspace{0.3cm}
\item \textbf{Estimation of the covariance}:\vspace{0.3cm}\\
For $t,s\in \pi_N$, we estimate the covariances \textcolor{black}{$R(t,s)$} by 
\begin{equation}
\label{eq:cov-proxy-estimation}
\hat{R}_J(t,s) = \frac{1}{|J|} \sum_{k\in J} X_{t,J}^{k,\text{prox}}X_{s,J}^{k,\text{prox}},
\end{equation}
where $X_{t,J}^{k,\text{prox}}$ denotes the proxy of the value $X_t^k$ associated to observations $Y_t^k$. 
\textcolor{black}{For a subset $\Delta \subset [0,T]$ and} for $\Delta\cap\pi_N\subset\pi_N$, let $\hat{\textbf{R}}_{N,\Delta}$ be the matrix collecting the covariance estimators $\hat{R}_J(t,s)$ for $t,s\in \Delta\cap\pi_N$ and let $\hat{\textbf{R}}_{N}:=\hat{\textbf{R}}_{N,\pi_N}$.
Similarly, for a fixed $t\in \pi_N$ and $\Delta\cap\pi_N\subset\pi_N$, we denote by 
$\hat{\textbf{b}}_{N,\Delta}(t)$ the vector consisting of covariance estimates $\hat{R}_J(t,s),s\in \Delta\cap\pi_N$.\vspace{0.3cm}
\item \textbf{Estimation of the predictors $\hat{X}^j_{t}$}:\vspace{0.3cm}\\
For a given $j$ and $t\in \pi_N \setminus \Delta_j$, approximate \eqref{eq:discrete-predictor} by
\begin{equation}
\label{eq:prediction-based-on-proxy}
\hat{X}_{t,N}^j = (a_j(t))'X_{\Delta_j,K}^{\text{prox}},
\end{equation}
where 
$$
a_j(t) = \hat{\textbf{R}}_{N,\Delta_j}^{-1}\hat{\textbf{b}}_{N,\Delta_j}(t)
$$
and $X_{\Delta_j,K}^{\text{prox}}$ is the vector containing the $X_{s,K}^{\text{prox}}$ for $s\in\Delta_j\cap\pi_N$ in the same order as in $\hat{\textbf{b}}_{N,\Delta_j}(t)$.
Note that here we have assumed explicitly that $\hat{\textbf{R}}_{N,\Delta_j}$ is of full-rank. 
Under assumptions of Theorem \ref{thm:discrete} this holds provided that, for every $(t,s)\in\pi_N^2$, $|J(t,s)|$ is large enough, ensuring $\hat{\textbf{R}}_{N}$ is close to $\textbf{R}_{N}$ (see Corollary \ref{cor:proxycov-convergence}). Note that, in practice, if $\hat{\textbf{R}}_{N,\Delta_j}$ is not of full-rank, then one could, for example, replace the inverses with general Moore-Penrose inverses.

\vspace{0.3cm}
\item \textbf{Estimation of $\rho(t,t|\Delta\cap\pi_N)$}:\vspace{0.3cm}\\
For $\Delta\cap\pi_N\subset\pi_N$, we define a centered Gaussian vector $Z_{N}$ with covariance $\hat{\textbf{R}}_{N}$. 
By simulating $M$ vectors $Z^m_{N},m=1,2,\ldots M$, we can define
 the estimator $\hat{\rho}_{M}(t,t|\Delta\cap\pi_N)$ by
\begin{equation}
\label{eq:rho-estimator}
\hat{\rho}_{M}(t,t|\Delta\cap\pi_N) = \frac{1}{M}\sum_{m=1}^M \left(Z_{t,N}^{m}-\hat{Z}^m_{t,N}\right)^2,
\end{equation}
where $\hat{Z}^m_{t,N}$ is the prediction of $Z^m_{t,N}$ computed through \eqref{eq:discrete-predictor}.\vspace{0.3cm}
\item \textbf{Approximation of the integral}:\vspace{0.3cm}\\
The final step of our construction requires estimating the integral in~\eqref{eq:non-gaussian-prediction}. For any continuous function $h \in L^1(\R,\phi(z)dz)$, let $L>0$ and approximate the integral 
$$
I_\phi(h) := \int_{-\infty}^\infty h(z)\phi(z)dz \qquad\textrm{ with }\qquad
I_{\phi,L}(h) = \int_{-L}^L h(z)\phi(z)dz.
$$
For practical or simulation purposes this can be further approximated by Riemann sums
$$
I_{\phi,L}(h) = \sum_{z_k} h(z_k)\phi(z_k)|z_k-z_{k-1}|,
$$
where the sequence $z_k$ is a partition of $[-L,L]$ such that $\max_k|z_k-z_{k-1}| \to 0$ as $L\to \infty$.\vspace{0.3cm}
\end{enumerate}

We are now ready to define the predictor for $Y_t^j$. 
\begin{dfn}
\label{def:non-gaussian-predictor}
Let $t\in \pi_N \setminus \Delta_j$. For given $L$ and $M$, define
\begin{equation}
\label{eq:non-gaussian-predictor-approx}
\hat{Y}^j_{t,L,M,K,N} = I_{\phi,L}\left(\hat{h}_{M}(t,\cdot)\right),
\end{equation}
where 
$$
\hat{h}_{M}(t,z) = \hat{f}_K\left(t,\hat{X}^j_{t,N}+z\sqrt{\hat{\rho}_{M}(t,t|\Delta_j\cap\pi_N)}\right),
$$
and $\hat{f}_K$, $\hat{X}^j_{t,N}$, and $\hat{\rho}_{M}$ are given by 
\eqref{eq:f-estimator}, \eqref{eq:prediction-based-on-proxy}, and \eqref{eq:rho-estimator}, respectively.
\end{dfn}
The following result shows that the prediction can, under our model assumptions, be done using the observed data only.
\begin{thm}
\label{thm:data-driven-predictor}
Suppose that the non constant random variable $Y_t$ satisfies Assumptions \eqref{ass:f-continuous} and \eqref{ass:var-X-known} and let $\hat{Y}_{t,L,M,K,N}$ be given by \eqref{eq:non-gaussian-predictor-approx}. Let $J^\ast=\min_{(t,s)\in\pi_N}|J(t,s)|$. Assume $\mathcal{F}_{\Delta\cap\pi_N}$ increases to $\mathcal{F}_{\Delta}$ as $N\to\infty$. Then, 
$$
\lim_{L\to\infty}\lim_{N\to\infty}\lim_{J^\ast\to \infty}\lim_{M\to\infty}\hat{Y}_{t,L,M,K,N} = \mathbb{E}\left[Y_t | \mathcal{F}_{\Delta}^Y\right]
$$
in probability.
\end{thm}
\begin{rem}
\changes{
Note that since $\mathcal{F}_{\Delta\cap\pi_N}$ increases to $\mathcal{F}_{\Delta}$ as $N\to\infty$, we have $|\pi_N|:=\max_{j}|t_j^N-t_{j_1}^N| \to 0$, where $t_{j-1}^N$ and $t_j^N$ are consecutive points in $\pi_N$. Indeed, information $\mathcal{F}_{\Delta\cap\pi_N}$ provided by discrete points can increase to the information $\mathcal{F}_{\Delta}$ based on observations in continuous time can only happen if either the resolution $|\pi_N|$ vanishes or the process $X$ is degenerate. This latter case is excluded due to non degeneracy assumption on the covariance function.
}
\end{rem}
\begin{rem}
	In the result above, note that $J^\ast\leq |K|$. The assumption $J^\ast\to\infty$ is required to ensure that $\hat{\mathbf{R}}_N\to \mathbf{R}_N$. 
\end{rem}

\begin{rem}
\changes{Typically, in the functional data analysis approach, the object of study is $Y_t = h(t) + \varepsilon_t$ for $t\in\Delta$ and $\varepsilon_t$ independent errors. 
The first standard aim, in order to apply functional techniques, is to reconstruct $h(t)$ from the discrete observations. 
The model studied in this paper allows for such reconstruction as one can set $Y_t = f(t,\varepsilon_t) = h(t) + \varepsilon_t$. 
In this case reconstruction of $h(t)$ is hidden into the estimation of $f$, and the underlying Gaussian process $X_t$ is determined by the error process $\varepsilon_t$. 
We remark, however, that if $\varepsilon_t$ is assumed to have independent errors for all $t\in [0,T]$ (as is often the case in FDA), then, formally, $\varepsilon_t$ corresponds to a white noise which is not a stochastic process \emph{per se}. 
On the other hand, if one requires independent errors only for certain discrete time points, one can overcome this issue by considering $\varepsilon_t = W_t - W_{t-\delta}$ for some $\delta$ small enough and with $W_t$ a Brownian motion. 
In this case, $\varepsilon_t$ exists as a Gaussian process and has independence on time points that are at least of distance $\delta$ away from each other.
}
\end{rem}
\begin{rem}
\changes{
Note that the discrete approximation described above is very general. 
If one further assumes that the processes are of a parametric nature (as in the simulations presented in Section~\ref{sec:simulations} below), then step (ii) can be simplified. Indeed, one can then estimate the relevant parameters and estimate the associated covariance structure by plugging-in the parameter estimates. 
}
\end{rem}

\subsection{On the rate of convergence}

Under additional assumption on the regularity of the function $f$, we now obtain an upper bound for the rate of convergence. 
In the sequel, for random variables $R_{\theta,\gamma}$, $R_{\theta}$ and $C_\theta$ and for a deterministic function $g(\gamma)$, we use the notation
$$
|R_{\theta,\gamma} - R_\theta| = O_p\left(C_\theta g(\gamma)\right),
$$
interchangeably with $C_\theta^{-1}|R_{\theta,\gamma} - R_\theta| = O_p\left(g(\gamma)\right)$.
That is, to indicate that, or any $\epsilon>0$ there exists $M>0$ such that 
$$
P(|R_{\theta,\gamma} - R_\theta|C_\theta^{-1}g^{-1}(\gamma) > M) < \epsilon.
$$
\begin{pro}
\label{prop:convergence-rate}
Suppose that $Y_t$ satisfies Assumptions \ref{ass:f-continuous} and \ref{ass:var-X-known} and let $\hat{Y}_{t,L,M,K,N}$ be given by \eqref{eq:non-gaussian-predictor-approx}. 
Let $J^\ast=\min_{(t,s)\in\pi_N^2}|J(t,s)|$. 
Assume $\mathcal{F}_{\Delta\cap\pi_N}$ increases to $\mathcal{F}_{\Delta}$ as $N\to\infty$. 
Now, let $t\in[0,T]$ be fixed.
Suppose further that the function $z\mapsto f(t,z)$ is continuously differentiable. 
Then, 
\begin{equation*}
\begin{split}
\hat{Y}_{t,L,M,K,N} -  \mathbb{E}\left[Y_t | \mathcal{F}_{\Delta}^Y\right] &= O_p\left(\frac{C_{L,N}}{\sqrt{K}}\right) + O_p\left(\frac{C_{L,N}}{\sqrt{M}}\right)\\
& + O_p\left(C_{L}\sqrt{\E\left(\hat{X}_{t,N,j}- \hat{X}^j_{t}\right)^2}\right) \\
&+ \int_{|z|>L}  f\left(t, \hat{X}^j_{t} + z \sqrt{\rho(t,t|\Delta)}\right) \phi(z)dz
\end{split}
\end{equation*}
in probability.
\end{pro}
\begin{rem}
The need for a compactification step, \changes{as per step 5 of the proposed algorithm}, also impacts the proposition above. Indeed, if one considers the whole real line ($L=\infty$), then the random constants $C_{L,N}$ in the first two terms does not necessarily remain bounded.  
\end{rem}
\begin{rem}
Under further assumptions on the function $f$, the last two terms in Proposition \ref{prop:convergence-rate} can be simplified further. 
By examining the proof, we observe that if $f$ is globally Lipschitz continuous (that is, the quantile function $Q_t$ is globally Lipschitz), then the third term can be replaced with 
$$
O_p\left(\sqrt{\E\left(\hat{X}_{t,N,j}- \hat{X}^j_{t}\right)^2}\right).
$$
Similarly, the last term 
$$
\int_{|z|>L}  f\left(t, \hat{X}^j_{t} + z \sqrt{\rho(t,t|\Delta)}\right) \phi(z)dz
$$
is related to the growth condition of $f$. In particular, if $f$ is also globally bounded, this term can be bounded by a Gaussian tail associated to the variance $\rho(t,t|\Delta)$. Note that in this case, the quantile function $Q_t$ is bounded as well, and consequently the empirical quantile function would converge uniformly on the whole interval $[0,T]$. This would allow us to drop compactification and set $L=\infty$ directly.
\end{rem}
\begin{rem}
Note that one would expect smaller errors when the prediction point $t$ is close to the observed part $\Delta^j$. However, as seen from Proposition \ref{prop:convergence-rate}, this does not affect the rate of convergence which arises mainly from estimation of the unknown function $f$ and the proxy values of the underlying Gaussian process. On the other hand, it has an effect on the underlying constants and on the value of $\rho(t,t|\Delta\cap\pi_N)$, cf. Remark \ref{rem:role-of-t}.
\end{rem}

\section{Numerical Study}
\label{sec:simulations}

In this section, we illustrate \changes{our methodology and} the theoretical results from Sections~\ref{sec:Gaussian-prediction} and \ref{sec:FDA-prediction}. \changes{We start with a simulation study aiming at illustrating the quality of the proposed reconstruction. Next, we illustrate the methodology on S$\&$P500 data.}

\subsection{Simulation study}
\label{subsec:simulations}

For the sake of simplicity, we assume throughout that $T=1$ and that the missing part of the processes is of length $0.2$. That is $\Delta=[0,t_U]\cup[t_L,1]$, with $t_L-t_U=0.2$.
The general setting considered here further assumes that the process $X_t$ is discretized at equidistant locations $t_i=\frac{i}{N}$, $i=1,\dots,N$ of $[0,1]$. 
We measure the loss incurred from estimating $X_{N}$ on $\bar \Delta= [0,1]\setminus\Delta$ with $\hat X_{N}$ via 
$$L(X_{N},\hat X_{N})=\frac{1}{0.2n}\sum_{t_i\in\bar \Delta}\big(X_{t_i}-\hat X_{t_i}\big)^2,$$
that is, the discretized approximation of the $L^2$ loss on $\bar \Delta$ based on the full processes $X$ and $\hat X$, $L(X,\hat X)=\int_{\bar \Delta}(X_t-\bar X_t)^2 dt.$ 
Different simulation settings are considered.

\emph{Simulation setting 1: fractional Brownian motions (fBm):} This setting assumes that $X_t=B^H_t$ is a fractional Brownian motion, that is, a centered Gaussian process with covariance 
$$
R(t,s) = \frac{1}{2}\left[t^{2H}+s^{2H}-|t-s|^{2H}\right].
$$
The Hurst parameter $H\in(0,1)$ gives the H\"older and self-similarity index of the process, allowing varying roughness and memory.
The case $H=1/2$ corresponds to (classical) Brownian motions. 
It is known that fractional Brownian motions satisfy Assumption \ref{assu:standing} for any $H\in(0,1)$, and their kernels $K_H(t,s)$ are completely known (see, e.g., \cite{Biagini-Hu-Oksendal-Zhang-2008} and references therein).

\emph{Simulation settings 2 and 3: bifractional Brownian motions (bBm): } A bifractional Brownian motion is a centered Gaussian process  with covariance
$$
R(t,s) = \frac{1}{2^K}\left[(t^{2H}+s^{2H})^K-|t-s|^{2HK}\right].
$$
The parameters are such that $H\in(0,1)$, $K\in (0,2)$ and $HK \in (0,1)$. 
Note that, for $K=1$, one recovers a fBm. 
These Gaussian processes were introduced to model situations where small increments tend to be stationary, but large increments do not. 
That is, only the small scale behavior tends to be close to a fBm. 
Moreover, the case $HK=\frac12$ provides an interesting model in mathematical finance, as then one obtains a process that has non-trivial quadratic variation similar to the standard Brownian motion, while possessing memory. 
For details on bifractional Brownian motion, see the monograph \cite{Tudor-2013}.
It is known that bifractional Brownian motions are purely non-deterministic and $HK$-self-similar. 
Consequently, Assumption \ref{assu:standing} is satisfied despite the fact that the kernel is not explicitly known. 
In the simulations, values $K=2/3$ (setting 2) and $K=3/4$ (setting 3) are considered. 

\emph{Simulation setting 4: independent mixed fractional Brownian motions (imfBm): } A mixed fractional Browian motion consists of a sum of a standard Brownian motion and a fractional Brownian motion, i.e. 
$$
X_t = W_t + B^H_t.
$$
In this setting, we assume that the processes $W$ and $B^H$ are independent, a case in which the covariance is known explicitly and is given by 
$$R(t,s)=\min(s,t)+\frac{1}{2}\left[t^{2H}+s^{2H}-|t-s|^{2H}\right].$$
Especially with $H>\frac12$, the model is widely used in mathematical finance. It is known that independent mixed fractional Brownian motions satisfy Assumption \ref{assu:standing}. 
See~\cite{Mishura-2008}.

Figure~\ref{fig:section2} displays, for each simulation setting and each $H\in\{0.2,0.5,0.8\}$ a random observation $X_t$ together with the optimal prediction $\hat X_{t,N}$ in the interval $\bar\Delta=[0.4,0.6]$. 
The reconstruction was conducted according to Theorem~\ref{thm:discrete}.
In particular, for illustration purposes, we assume that the covariance $R_N$ is known and is used in the reconstruction of $X_{t,N}$ on $\bar\Delta$. 

As evidenced by Figure~\ref{fig:section2}, the curves become smoother for larger self-similarity indices. 
Furthermore, when the self-similarity index is $0.5$ (that is, in cases 1 and 4, for $H=0.5$), the optimal reconstruction yields straight lines, as displayed. 
The other cases show the impact of the memory present in the processes on the optimal prediction. 
Note again that, in the context of Theorem~\ref{thm:discrete}, $\hat X_{t,N}$ is based on the observed parts of $X$ only and that the reconstruction can be conducted independently for each curve. 


\begin{figure}[h!]
  \includegraphics[width=\textwidth]{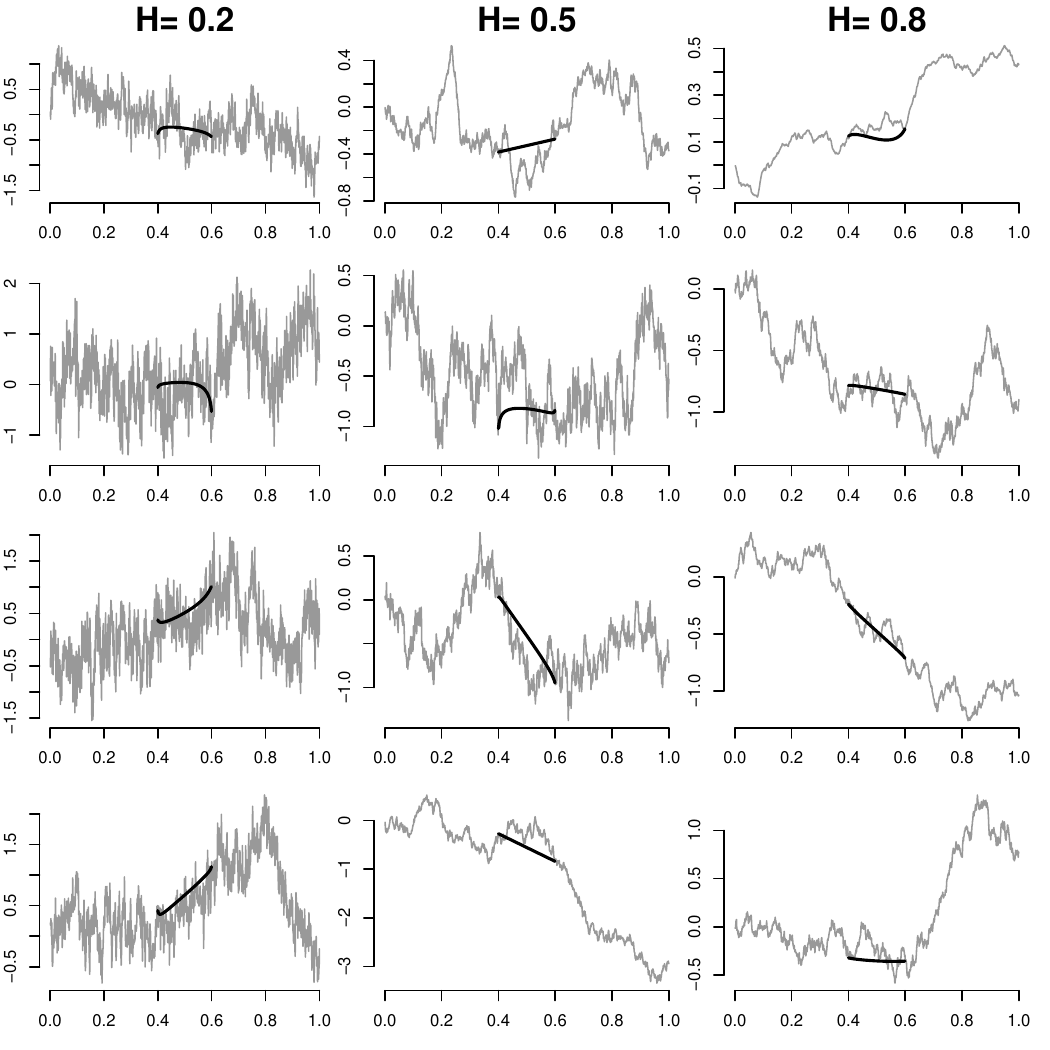}
  \caption{Gaussian processes (in grey) generated according to each simulation setting (row 1: fBm, row 2-3 : bBm, for $K=2/3$ and $K=3/4$, row 4: imfBm) for Hurst indices $H=0.2$ (left), $H=0.5$ (middle) and $H=0.8$ (right). The optimal reconstructed curves on $\bar\Delta=[0.4,0.6]$ are displayed in black. 
}
\label{fig:section2}
\end{figure}

In order to assess the quality of the reconstruction, the following experiment was conducted. 
In each simulation setting, $n=500$ observations $X_1,\dots,X_n$ were generated and observed on the regular grid $\pi_{N}$ for $N=250$. 
For each $j=1,\dots,500$, $Y_j=f(X_j)$ was assumed to be observed on $\Delta_j=[0,1]\setminus [t_{U,j},t_{L,j}]$, with $t_{U,j}$ chosen uniformly within the interval $[0.1,0.7]$ and $t_{L,j}=t_{U,j}+0.2$. 
This resulted in uniformly missing intervals of length $0.2$ within $[0.1,0.9]$. The functions used in this simulation setting were $f(X_j)=X_j$ and $f(X_j)=\exp(X_j)$ (referred henceforth as the identity and exponential case, respectively).

We compared our approach to several classical curve reconstruction techniques. 
More precisely, the competitors considered were:
\begin{itemize}
\item spline smoothing (using $9$ basis splines of order $2$);
\item Fourier smoothing (using $13$ basis elements); and
\item linear interpolation.
\end{itemize}
Note that, for spline and Fourier smoothing, the number of basis elements were based on cross-validation. 
For our approach, we consider three scenarios. 
In the first two, the parametric family of the Gaussian processes are known. 
The first scenario assumes that both $f$ and $H$ are fixed, so that, we can use Theorem~\ref{thm:discrete} as a baseline. 
In the second, we assume that $f$ is known. 
The Hurst index $H$ is estimated as in~\cite{Coeurjolly2000} based on the observed data points and used in the estimation of the covariance matrix in step (ii). 
Finally, the last scenario uses the algorithm detailed in Section 3, without assuming neither knowledge on $f$, $H$ nor the process distribution. 

The following figures display boxplots of the losses $L(Y_{j,N},\hat Y_{j,N})$ for each simulation setting and for each value of $H\in\{0.2,0.5,0.8\}$. The cases $f(X)=X$ and $f(X)=\exp(X)$ are displayed in Figures~\ref{fig:section3} and~\ref{fig:section3exp}, respectively.
The predictions are displayed as follows: spline smoothing (S9) in red, Fourier smoothing (F13) in green, linear interpolation (I) in yellow and our proposals in different hues of blue. 
The case where the Gaussian process is fully known or when $H$ is estimated are denoted by (G) and (GH) respectively. 
The cases where the underlying processes are unknown and where $f$ is estimated are denoted by (IdH) and (ExpH) (in the identity and exponential case, respectively). 

Simulation results clearly highlight the optimality of the proposed reconstruction. 
Remarkably, the cost of not knowing $H$ seems to remain negligible. 
In the identity case, the estimation of $f$ does not change the precision much. 
The case $f=\exp$ seems to be slightly more unstable, as the function has to be estimated pointwise. 
Note that, the algorithm still provides competitive estimation of the missing parts. 
Naturally, if the type of the underlying parametric process is known, it is advisable to avoid estimating the full function $f$ and restrict to the estimation of the model parameters. 
In our examples, larger values of the Hurst index $H$, corresponding to processes with stronger memory, yield more accurate reconstruction for all techniques. 
In the cases where the optimal reconstruction is the linear interpolation (simulation settings 1 and 4, for $H=0.5$ and $f=\textrm{id}$), our approach and the linear interpolation indeed coincide. 
In all other cases, there is a clear gain in using the optimal approach. 
The spline and Fourier smoothing do poorly in general, and get worse for lower $H$ values. 
This is explained as these interpolations are known to accurately approximate functions only in their observed parts and lower $H$ values provide rougher curves. 

\subsection{Real data illustration}

\changes{Next, we illustrate the curve reconstruction on real data. 
We analyze the returns $Y=(Y_1,\dots,Y_T)$, of the S$\&$P500 index from February 1st, 2015 to February 1st, 2017. 
Without loss of generality, data was standardized so that the initial value is set at $0$ and has globally empirical variance $1$. 
Likewise, time was coded in such a way that the data is observed on the interval $[0,1]$.
The final dataset contains $T=542$ such returns. 
}

\changes{
For $R=100$ replications, we considered $\Delta_r=[0,t_{U,r}]\cup[t_{L,r},1]$, with $t_{U,r}$ selected uniformly at random in $[0.1,0.7]$ and $t_{L,r}=t_{U,r}+0.2$ ($r=1,\dots,R$). 
We compare the discretized $L^2$ losses $L(Y,\hat Y_{r})$ of the same reconstruction methods as in Section~\ref{subsec:simulations}, namely: (i) interpolation, (ii) spline smoothing, (iii) Fourier smoothing and (iv) the method proposed in this paper.
We explore reconstruction under the parametric assumption that the data comes from each of the four simulation settings described in the previous section.
Note that the use of fractional processes to study the S$\&$P500 index has been discussed in the literature. See, for example, \cite{DomRiv2011,Gat2018}.
Remark also that estimating the full model (as in Section 3, with a nonparametric link function $f$) is not possible given that we only have one functional observation.
Both smoothing methods were optimized by crossvalidation (over the number of basis elements ($k\in\{5,\dots,25\}$) and, in the first case, the order of the splines ($\ell\in\{2,\dots,6\}$)). The resulting choices are $k=11$ and $\ell=4$, and $k=11$ for the spline and Fourier cases, respectively.
One example of such reconstructions is available in Figure~\ref{fig:SP500} below. 
Table~\ref{table:SP500} provides the average losses over the $R$ replications. 
In particular, it shows that all parametric reconstructions perform better than smoothing. In particular, the optimal reconstruction under bifractional Brownian motions (with $K=3/4$) improves the interpolation loss by $20\%$.
}

\begin{figure}[h!]
  \includegraphics[width=0.8\textwidth]{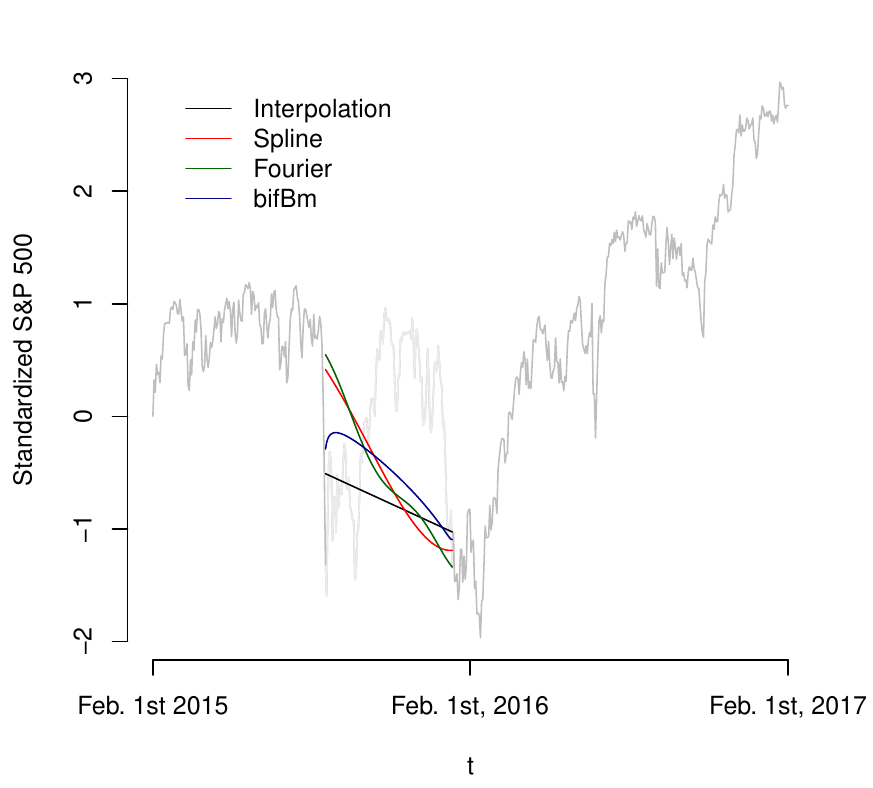}
  \caption{\changes{Standardized S$\&$P500 data. Four reconstructions are provided (black: interpolation, red: spline, green: Fourier, blue: optimal reconstruction for a bifractional Brownian motion ($K=3/4$). The missing period ranges from Aug. 24th 2015 to Jan. 13th, 2016. In this particular case, the losses are, respectively, $1.114$, $1.336$, $1.315$ and $0.893$.}}
\label{fig:SP500}
\end{figure}

\begin{table}[hbt]
\caption{\changes{Average $L^2$ loss values over $R=100$ replications of various reconstruction methods (see text).}}
\begin{tabular}{c|c|c|c|c|c|c}
\label{table:SP500}
   Interpolation & Spline & Fourier & Fractional & Bifractional & Bifractional & Independent mixed  \\
&&&&($K=2/3$)&($K=3/4$)&\\
    \hline
  $0.7582$  & $1.4771$ & $3.8259$ & $0.677$ & $0.6412$ & $0.6017$ & $0.6763$
  \end{tabular}
\end{table}

\begin{figure}[h!]
  \includegraphics[width=0.9\textwidth]{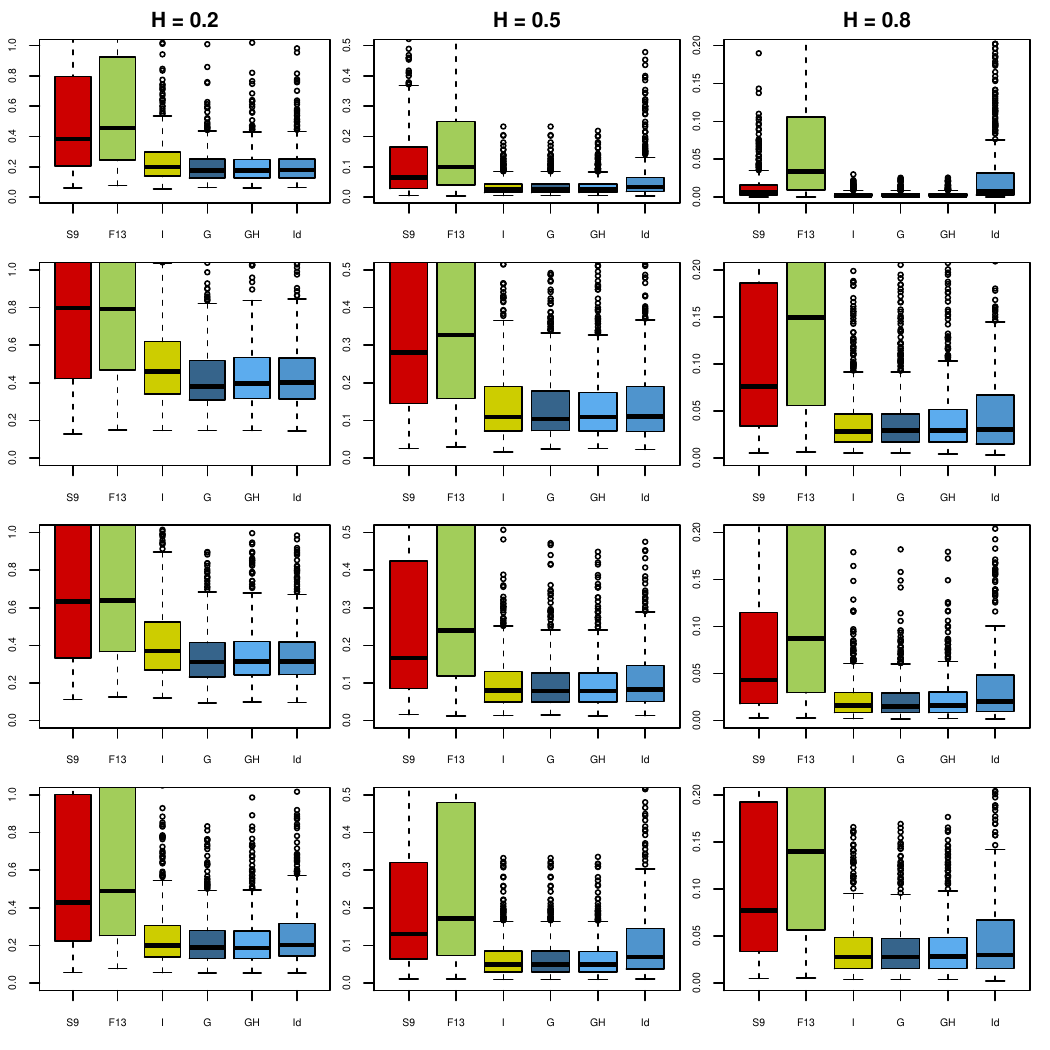}
  \caption{Boxplots of the $L^2$ loss, for $f=\textrm{id}$, for the four simulations settings described above (one per row), for $H=0.2$, $0.5$ or $0.8$ (per column). The reconstruction methods are, from left to right: (red) spline smoothing using 9 basis elements of order 2 (S9); (green) Fourier smoothing using 13 basis elements (F13); (yellow) the reconstruction obtained by linearly interpolating $X_{t_U}$ to $X_{t_L}$; (blue) our proposal assuming the process is Gaussian with the right $H$ value (left, G), with an estimated $H$ (middle, GH) or with an estimated $f$ and $H$ (right, IdH).}
\label{fig:section3}
\end{figure}

\begin{figure}[h!]
  \includegraphics[width=0.9\textwidth]{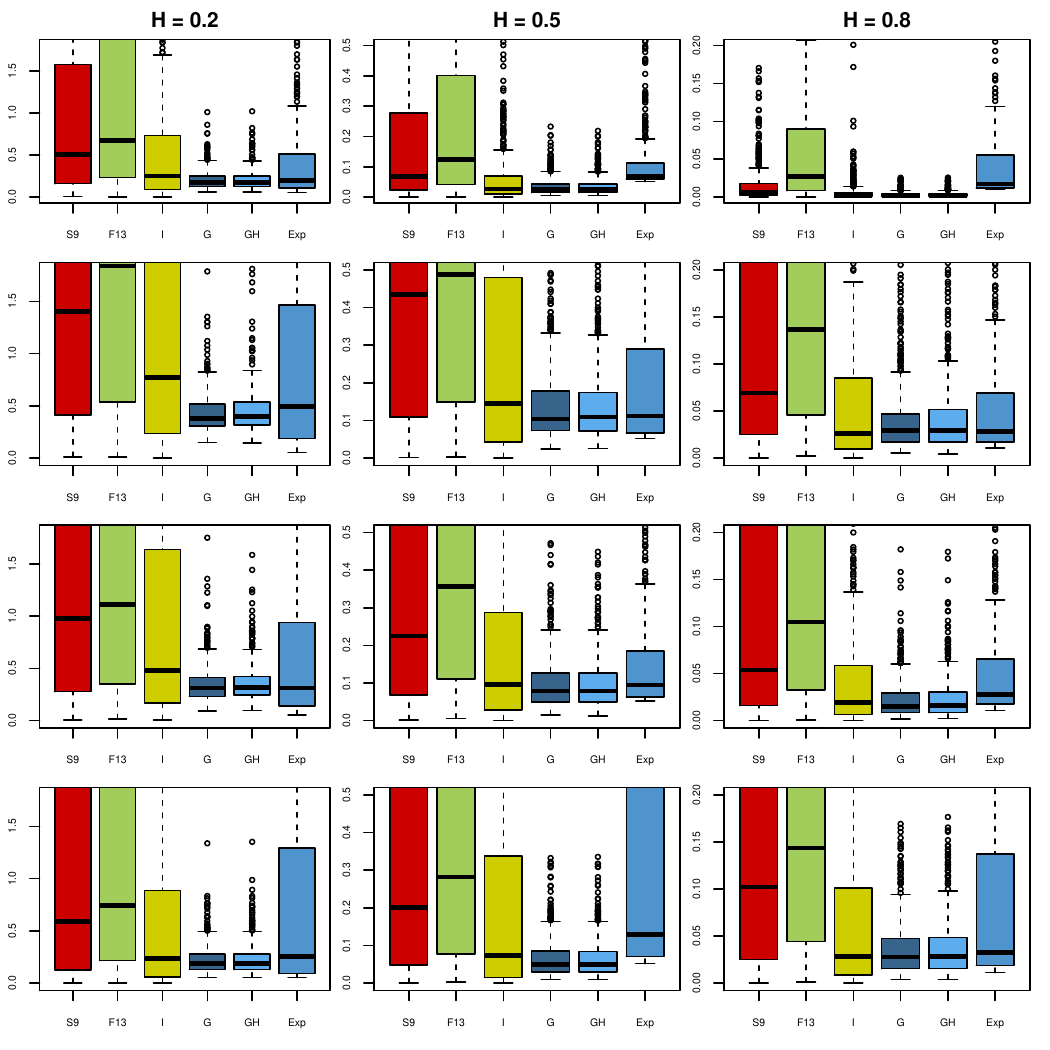}
  \caption{Boxplots of the $L^2$ loss, for $f=\exp$, for the four simulations settings described above (one per row), for $H=0.2$, $0.5$ or $0.8$ (per column). The reconstruction methods are, from left to right: (red) spline smoothing using 9 basis elements of order 2 (S9); (green) Fourier smoothing using 13 basis elements (F13); (yellow) the reconstruction obtained by linearly interpolating $X_{t_U}$ to $X_{t_L}$; (blue) our proposal assuming the process is Gaussian with the right $H$ value (left, G), with an estimated $H$ (middle, GH) or with an estimated $f$ and $H$ (right, IdH).}
\label{fig:section3exp}
\end{figure}

\appendix 
\section{Preliminaries on Gaussian analysis}
\label{sec:appendix-Gaussian}

In this section we recall some preliminaries on Gaussian analysis and processes. For details on the topic, we refer to \cite{Janson-1997,Nualart-2006}; and \cite{Sottinen-Viitasaari-2016b}.

Suppose that $X$ is a centered and separable Gaussian process on some compact interval $[0,T]$. It is known \cite{Sottinen-Viitasaari-2016b} that if the covariance $R$ of $X$ is of trace class, i.e. $R(t,t) = Var(X_t) \in L^2([0,T])$, then $X$ admits a Fredholm representation \eqref{eq:fredholm}
\begin{equation*}
X_t = \int_0^T K(t,s)dW_s,
\end{equation*} 
where $W$ is a standard Brownian motion and $K \in L^2([0,T]^2)$. In particular, this is the case whenever $X$ is continuous.
\begin{dfn}[Operator associated with a kernel]\label{dfn:associated}
The associated operator to a kernel $\Gamma\in L^2([0,T]^2)$, also denoted by $\Gamma$, is defined as $\Gamma:L^2([0,T])\rightarrow L^2([0,T])$ with
\begin{equation}
\label{eq:associated-operator}
(\Gamma f)(t) = \int_0^T f(s)\Gamma(t,s)\, ds. 
\end{equation} 
\end{dfn}

\begin{dfn}[Isonormal process]\label{dfn:isonormal}
The isonormal process associated with $X$, also denoted by $X$, is the Gaussian family $\{X(h), h\in\H\}$, where the Hilbert space $\H=\H(R)$ is the closed span of $\{\1_t := \1_{[0,t)},\ t\le T\}$, closed for the inner product on $\mathcal{H}$ defined via $\la \1_t,\1_s\ra_{\H} := R(t,s)$.  
\end{dfn}
By Definition \ref{dfn:isonormal}, $X(h)$ can be viewed as the image of $h\in\H$ in the isometry that extends the relation
$$
X\left(\1_t\right) := X_t
$$
linearly. This allows us to define the Wiener integral with respect to $X$.

\begin{dfn}[Wiener integral]\label{dfn:wi}
$X(h)$ is the \emph{Wiener integral} of the element $h\in\H$ with respect to $X$, which is denoted by
$$
\int_0^1 h(t)\, dX_t := X(h).
$$   
\end{dfn}
\begin{rem}
Note that, it may happen that the space $\H$ is not a space of functions, but instead it has to be considered as a space of generalised functions, see~\cite{Pipiras-Taqqu-2001}. 
\end{rem}
\begin{dfn}\label{dfn:adjoint-associated}
The adjoint associated operator $\Gamma^*$ of a kernel $\Gamma\in L^2([0,T]^2)$ is defined by linearly extending the relation
$$
\Gamma^*\1_t = \Gamma(t,\cdot).
$$
\end{dfn}
\begin{exa}
Let $h$ be a step function of the form
$$
h(s) = \sum_{k=1}^n \alpha_k \1_{(t_{k-1},t_k]},
$$
where $\alpha_1,\ldots,\alpha_n \in \R$ and $0\leq t_0 < t_1<\ldots t_n\leq T$. Then
$$
\Gamma^* h(s) = \int_0^T h(t)\Gamma(dt,s).
$$ 
\end{exa}
Definition \ref{dfn:adjoint-associated} allows us to provide transfer principle for Wiener integrals. For details, see \cite{Sottinen-Viitasaari-2016b}.
\begin{thm}[Transfer principle for Wiener integrals]\label{thm:transfer-principle-wi}
Let $X$ be a separable centered Gaussian process with representation \eqref{eq:fredholm} and let $h\in \H$. Then,
$$
\int_0^T h(t)\, dX_t = \int_0^T K^* h (t)\, dW_t.
$$
\end{thm}
Finally, we recall two important spaces related to a Gaussian process $X$.
\begin{dfn}
\label{def:first-chaos}
The first chaos of $X$, denoted by $H^X =H^X_T$, is defined as the Hilbert space spanned by $X_s,s\in[0,T]$ and closed for the $L^2(\Omega)$-distance. 
Similarly, $H^X_t$ denotes the Hilbert space spanned by $X_s,s\in[0,t]$.
\end{dfn}
\begin{dfn}[Cameron-Martin space]
The Cameron-Martin space of $X$, denoted by $CM_X$, is the space of functions defined by
$$
f(t) = \E[X_t \xi],\quad \xi\in H^X.
$$
\end{dfn}
The Cameron-Martin space is a Hilbert space when equipped with the inner product
$$
\langle f_1,f_2\rangle_{CM_X} = \E[\xi_1\xi_2],
$$
where $f_i(\cdot)=\E[X_\cdot \xi_i],i=1,2$. 
Moreover, all the spaces $\mathcal{H}$, $H^X$, and $CM_X$ are isometric and separable whenever $X$ is separable. 
More specifically, the isometry between $\mathcal{H}$ and $H^X$ is given by the relation 
$$
\int_0^T h(t)dX_t \in H^X,\quad h \in \mathcal{H}.
$$ 
Note that the Cameron-Martin space is also known as a Reproducing Kernel Hilbert Space (RKHS). 
In this instance, the covariance is the reproducing kernel, i.e., is defined via $f(t) = \langle f, R(t,.) \rangle_{RKHS}$.

\section{Proofs}
\label{sec:proofs}
\subsection{Proofs related to Section \ref{sec:Gaussian-prediction}}
For a linear operator $A:X\mapsto Y$, where $X$ and $Y$ are some vector spaces, the \emph{null space} and \emph{range} are defined by
$$
\mathcal{N}(A) = \{x \in X: Ax = 0\} \subset X
\quad \textrm{ and }\quad 
\mathcal{R}(A) = \{y \in Y: Ax = y \text{ for some }x\in X\} \subset Y.
$$
If $\mathcal{N}(A) =\{0\}$, then the inverse operator $A^{-1}:\mathcal{R}(A) \mapsto X$ exists on $\mathcal{R}(A)$. In particular, if $\mathcal{R}(A) = Y$, then the inverse $A^{-1}$ exists on the entire domain $Y$. The following non-trivial result, originally due to Banach, is applied to prove Theorem \ref{thm:K-invertible}.
\begin{thm}
\label{thm:inverse-operator-bounded}
Let $A:X\mapsto Y$ be a bounded linear operator such that $\mathcal{N}(A) =\{0\}$ and $\mathcal{R}(A) = Y$. Then the inverse operator $A^{-1}$ is linear and bounded as well.
\end{thm}
\begin{cor}
\label{cor:operator-solution}
Let $A:X\mapsto Y$ be a bounded linear operator such that $\mathcal{N}(A) =\{0\}$ and $\mathcal{R}(A) = Y$. Then for any $y\in Y$ the equation 
$$
Ax=y
$$
has a unique solution $x\in X$.
\end{cor}

In order to prove Theorem~\ref{thm:Gaussian-fredholm-prediction}, we need the following preliminary results. 

\begin{thm}
\label{thm:K-invertible}
Let $X$ be given by \eqref{eq:fredholm} and suppose further that $H^X_T = H^W_T$. 
Then, the associated operator $K: L^2([0,T]) \mapsto L^2([0,T])$ satisfies $\mathcal{N}(K) = \{0\}$ and $\mathcal{R}(K) = CM_X \subset L^2([0,T])$. In particular, $K:L^2([0,T]) \mapsto CM_X$ is a bounded linear operator that has a bounded and linear inverse $K^{-1}:CM_X \mapsto L^2([0,T])$.  
\end{thm}
\begin{proof}
The fact that $H^X_T = H^W_T$ ensures that the operator $K^*: \mathcal{H} \mapsto L^2([0,T])$ is bijective. 
Indeed, each $\xi \in H^X_T$ satisfies
$$
\xi = \int_0^T f(s) dX_s = \int_0^T (K^*f)(s)dW_s = \int_0^T g(s)dW_s,
$$
for some $g \in L^2([0,T])$. 
Conversely, each element $\phi \in H^W_T$, with representation
$$
\phi = \int_0^T h(s)dW_s, \quad h \in L^2([0,T]),
$$ 
belongs to $H^X_T$ as well. Thus for each $f\in L^2([0,T])$ there exists $\xi \in H^X_T$ such that
$$
(Kf)(t) = \int_0^T K(t,s)f(s)ds = \E\left[X_t \xi\right].
$$
This gives directly that $\mathcal{R}(K) = CM_X$, and that if $Kf \equiv 0$, then 
$$
\xi = \int_0^T f(s)dW_s = 0 
$$
almost surely, implying $f\equiv 0$. Thus $\mathcal{N}(K) = \{0\}$. Moreover, it is clear from the isometry 
$$
\Vert Kf\Vert_{CM_X} = \Vert f\Vert_{L^2([0,T])} 
$$
that $K:L^2([0,T]) \mapsto CM_X$ is a bounded linear operator. 
Theorem \ref{thm:inverse-operator-bounded} concludes the proof.
\end{proof}
\begin{rem}
Since $K^*$ is invertible, we can transform our prediction into a Wiener integral with respect to $X$ by using
\begin{equation}
\label{eq:wiener-wrt-X}
\hat{X}_t = \int_0^T f_t(s)dW_s = \int_0^T \left(\left(K^*\right)^{-1}f_t\right)(s)dX_s.
\end{equation}
However, even in some well-studied cases such as the fractional Brownian motion, the expression~\eqref{eq:wiener-wrt-X} is impractical as the operators $K^*$ and $\left(K^*\right)^{-1}$ might be complicated.
\end{rem}
\begin{cor}
\label{cor:unique-solution-CM}
Let $X$ be given by \eqref{eq:fredholm} and suppose further that $H^X_T = H^W_T$. 
Then, for any $g\in CM_X$, the integral equation of the first kind
$$
(Kf)(t) = \int_0^T K(t,s)f(s)ds = g(t) 
$$
has a unique solution $f\in L^2([0,T])$.
\end{cor}
\begin{proof}
The claim follows directly from Theorem \ref{thm:K-invertible} and Corollary \ref{cor:operator-solution}.
\end{proof}
We are now in position to prove our first main theorem.
\begin{proof}[Proof of Theorem \ref{thm:Gaussian-fredholm-prediction}]
Let $t\in [0,T]\setminus \Delta$. 
It is well-known that $\hat{X}_t$ is Gaussian. 
Moreover, since $\hat{X}_t \in H_T^X$, the assumption $H_T^X =H_T^W$ implies that the representation
$$
\hat{X}_t = \int_0^T f_t(s) dW_s
$$
holds. 
By definition of the conditional expectation, which is an orthogonal projection, $\hat{X}_t- X_t$ is orthogonal to all random variables of the form 
\begin{equation}
\label{eq:Y}
Y = \sum_{k=1}^n \alpha_k X_{t_L^k} + \sum_{k=1}^n \beta_k (X_{u_k}-X_{t^k_L}),
\end{equation}
where $u_k\in[t_L^k, t^k_U]$ for all $k=1,\ldots,n$. 
By linearity, it follows that $\hat{X}_t-X_t$ is orthogonal to all Gaussian random variables measurable with respect to $\mathcal{F}_\Delta$, which in turn implies, by Gaussianity, that $\hat{X}_t-X_t$ is orthogonal to $L^2(\Omega,\mathbb{P},\mathcal{F}_\Delta)$. 

Using~\eqref{eq:fredholm} in \eqref{eq:Y} and the fact that $
\hat{X}_t - X_t = \int_0^T f_t(s)-K(t,s)dW_s
$, the orthogonality condition gives
\begin{equation*}
\begin{split}
0 & = \int_0^T \left[f_t(s)-K(t,s)\right]\sum_{k=1}^n \left[\alpha_k K(t_L^k,s) + \beta_k\left(K(u_k,s)-K(t_L^k,s)\right)\right]ds.
\end{split}
\end{equation*}
As this holds for all $\alpha_k,\beta_k \in \mathbb{R}$, we obtain that, for all $k=1,\ldots,n$ and $t_L^k\leq u_k\leq t_U^k$, 
$$
\int_0^T f_t(s)K(u_k,s)ds = \int_0^T K(t,s)K(u_k,s)ds.
$$
This shows that~\eqref{eq:fredholm-system} holds. 

For any $g_\xi$ such that $\int_0^T g_\xi(s)K(u_k,s)ds = 0$, the same reasoning shows that $\int_0^T g_\xi(s)dW_s$ is orthogonal to $\mathcal{F}_\Delta$. 
The extra condition $\int_0^T g_\xi(s)f_t(s)ds = 0$ ensures that $\int_0^T f_t(s)dW_s$ is measurable with respect to $\mathcal{F}_\Delta$. 
The uniqueness of $f_t$ then follows from the uniqueness of the conditional expectation $\hat{X}_t = \int_0^T f_t(s)dW_s$.
This completes the proof.
\end{proof}

\begin{proof}[Proof of Theorem \ref{thm:Gaussian-regular-law}]
It is known that $X|\mathcal{F}_\Delta$ is Gaussian with Gaussian random mean and deterministic covariance given by
$$
\hat{X}_t = \E[X_t | \mathcal{F}_{\Delta}]\quad\textrm{ and }\quad
\rho(t,s|\Delta) = \E\left[(X_t -\hat{X}_t)(X_s-\hat{X}_s)\right],
$$
respectively. 
See, e.g. \cite[Chapter 9]{Janson-1997}.
The mean is provided in Theorem \ref{thm:Gaussian-fredholm-prediction}. Similarly, the representation \eqref{eq:rho} for $\rho$ follows from 
$$
X_t - \hat{X}_t = \int_0^T K(t,u)-f_t(u)dW_u
$$
and the It\^o isometry
$$
\textrm{Cov}\left(\int_0^T f_1(s)dW_s,\int_0^T f_2(s)dW_s\right) = \int_0^T f_1(s)f_2(s)ds.
$$
\end{proof}
\begin{proof}[Proof of Theorem \ref{thm:Gaussian-prediction}]
By Theorem \ref{thm:Gaussian-fredholm-prediction}, for any $1\leq k\leq n$ and $t_L^k\leq u_k\leq t_U^k$, it holds that
$$
\int_0^T f_t(s)K(u_k,s)ds = \int_0^T K(t,s)K(u_k,s)ds.
$$
With the convention $t_U^{0}=0$, we have 
$$
\int_{t_U^{k-1}}^{t_L^k}f_t(s) K(t_L^k,s)ds = \int_0^{t_L^k} K(t,s)K(t_L^k,s)ds - \int_0^{t_U^{k-1}} f_t(s)K(t_L^k,s)ds
$$
and
$$
\int_{t_L^k}^{u_k} [f_t(s)-K(t,s)]K(u_k,s)ds = \int_0^{t_L^k} [K(t,s)-f_t(s)]K(u_k,s)ds.
$$
The result follows.
\end{proof}
\begin{proof}[Proof of Corollary \ref{cor:one-missing}]
In the setting of this corollary, equations \eqref{eq:thm-volterra-eq} and \eqref{eq:thm-boundary-eq} rewrite
\begin{equation}
\label{eq:proof-1}
\int_0^u \left[f_t(s) - K(t,s)\right]K(u,s)ds = 0,\quad 0\leq u\leq t_U,
\end{equation}
\begin{equation}
\label{eq:proof-2}
\int_{t_U}^{t_L} f_t(s) K(t_L,s)ds = \int_0^{t_L}K(t,s)K(t_L,s)ds - \int_0^{t_U}f_t(s)K(t_L,s)ds,
\end{equation}
and
\begin{equation}
\label{eq:proof-3}
\int_{t_L}^u \left[f_t(s)-K(t,s)\right]K(u,s)ds = \int_0^{t_L} \left[K(t,s)-f_t(s)\right]K(u,s)ds,\quad t_L\leq u \leq T.
\end{equation}
The solution to \eqref{eq:proof-1} is given by $f_t(s) = K(t,s),s\leq t_U$. 
It is unique since $H^X_t = H^W_t$ allows to use Corollary \ref{cor:unique-solution-CM}. \eqref{eq:proof-2} therefore rewrites
$$
\int_{t_U}^{t_L} f_t(s) K(t_L,s)ds = \int_0^{t_U}K(t,s)K(t_L,s)ds.
$$
Let $f_t(s) = \tilde{f}_t(s) + c(t)K(t_L,s)$. 
It follows from \eqref{eq:c} that
$$
\int_{t_U}^{t_L}\tilde{f}_t(s)K(t_L,s)ds = 0
$$
and since $H_t^X = H_t^W$, we must have $\tilde{f}_t(s) = 0$ for all $t_L<s<t_U$. Indeed, otherwise $\int_0^Tf_t(s)dW_s \in H_s^W\setminus H_{t_L}^W= H_s^X\setminus H_{t_L}^X$ for some $t_L<s<t_U$, and hence $\hat{X}_t$ would not be measurable. Finally, \eqref{eq:int-eq} follows by plugging $f_t(s)$ for $s<t_L$ into \eqref{eq:proof-3}.
\end{proof}
\begin{proof}[Proof of Corollary \ref{cor:future-prediction}]
Theorem \ref{thm:Gaussian-fredholm-prediction} implies that $\hat{X}_t = \int_0^T f_t(s)dW_s$, where $f_t$ satisfies, for $0\leq u \leq t_U$,
$$
\int_0^u f_t(s)K(u,s)ds = \int_0^u K(t,s)K(u,s)ds.
$$
Since $H^X_t = H^W_t$ for all $t\in[0,T]$, Corollary \ref{cor:unique-solution-CM} implies that the unique solution is given by $f_t(s) = K(t,s)$.
\end{proof}

\begin{proof}[Proof of Theorem \ref{thm:discrete}]
Let $t\in [0,T]\setminus \Delta_N$. 
As in the proof of Theorem \ref{thm:Gaussian-prediction}, the random variable $\hat{X}_{t,N}$ is Gaussian and measurable with respect to $\Delta_N$. 
Thus, it has the representation \eqref{eq:discrete-predictor} and $\hat{X}_{t,N}-X_t$ is orthogonal to all random variables 
$$
Y = \sum_{t_k\in \Delta_N} \alpha_k X_{t_k},\quad \alpha_k \in \R.
$$
It follows that 
$$
0 = \E\left[(\hat{X}_{t,N}-X_t)Y\right] = \sum_{t_k\in \Delta_N} \alpha_k \E\left[(\hat{X}_{t,N}-X_t)X_{t_k}\right]
$$
and hence
$$
\E\left[(\hat{X}_{t,N}-X_t)X_{t_k}\right] = 0.
$$
Now representation \eqref{eq:discrete-predictor} gives
$$
\E\left[\hat{X}_{t,N}X_{t_k}\right] = \sum_{t_j\in \Delta_N} a_j(t) (\textbf{R}_N)_{jk}
$$
which gives us \eqref{eq:vector-a}. In order to prove \eqref{eq:discrete-predictor-convergence}, we set 
$$
M_N = \E\left[X_t | \mathcal{F}_{\Delta_N}\right].
$$
Since $\mathcal{F}_{\Delta_N} \subset \mathcal{F}_{\Delta_{N+1}}$, it follows that $M_N$ is a martingale (with respect to $\left(\mathcal{F}_{\Delta_N}\right)_{N\geq 1}$). By Jensen's inequality and the fact that $X_t$ has bounded variance, we get
$$
\sup_{N\geq 1}\E M_N^2 \leq \E X_t^2 < \infty. 
$$
Hence $M_N$ is uniformly bounded in $L^2(\Omega)$. Thus we may apply the martingale convergence theorem giving us $M_N \to M_\infty$ almost surely and in $L^2(\Omega)$. 
From this and the fact that $\mathcal{F}_{\Delta_N}$ increases to $\mathcal{F}_\Delta$ it follows that 
$$
M_\infty = \hat{X}_t = \E\left[X_t | \mathcal{F}_\Delta\right].
$$
This concludes the proof.
\end{proof}
\subsection{Proofs related to Section \ref{sec:FDA-prediction}}
\begin{proof}[Proof of Theorem \ref{thm:prediction-FDA}]
Since $x\mapsto f(t,x)$ is bijective for all $t\in[0,T]$, it follows that $\mathcal{F}_\Delta^Y = \mathcal{F}_\Delta^X$. The result now follows by writing
$$
f(t,X_t) = f(t,\hat{X}_t + X_t - \hat{X}_t)
$$
which yields, by the independence of $\hat{X}_t$ and $X_t - \hat{X}_t$ and the fact that $X_t - \hat{X}_t \sim N(0,\rho(t,t|\Delta))$, that
$$
\mathbb{E}\left[f(t,X_t) | \mathcal{F}_\Delta\right] = \int_{-\infty}^\infty f(t,\hat{X}_t + z)\phi_{\rho(t,t|\Delta)}(z)dz,
$$
where $\phi_{\rho(t,t|\Delta)}$ denotes the density function of $N(0,\rho(t,t|\Delta))$. The claim now follows from a change of variable.
\end{proof}
\begin{pro}
\label{pro:proxy-convergence}
Let $t\in \pi_N$ and let ${X}_{t,K}^{\text{prox}}$ be given by \eqref{eq:proxy-def}. Then, for any $p\geq 1$, we have, almost surely and in $L^p(\Omega)$  
\begin{equation}
\label{eq:proxy-converges}
{X}_{t,K}^{\text{prox}} \to X_{t}
\end{equation}
as $|K| \to \infty$. 
\end{pro}
\begin{proof}
For simplicity, we suppress $t$ from the notation and simply write e.g. ${X}^{\text{prox}}_K$ instead of ${X}_{t,K}^{\text{prox}}$. With this simplified notation, 
$$
{X}^{\text{prox}}_K =\Phi^{-1}\left(\hat{F}_{K}(Y)\right).
$$
First we observe that, by the Glivenko-Cantelli theorem and continuity of $\Phi^{-1}$, the convergence \eqref{eq:proxy-converges} holds almost surely. 
To show the $L^p(\Omega)$ consistency, it is sufficient to show that 
$$
\sup_{|K|\geq 1} \E \left|{X}^{\text{prox}}_K\right|^p < \infty.
$$
Indeed, this implies that, for any $q< p$, the family
$$
\left|{X}^{\text{prox}}_K\right|^q,\qquad |K|=1,2,\ldots
$$
is uniformly integrable. 
This implies that \eqref{eq:proxy-converges} holds in $L^q(\Omega)$. Since $p$ is arbitrary, the result would hold. 

In the sequel, we assume without loss of generality that the bijective function $f$ is strictly increasing (the symmetric case with decreasing $f$ can be tackled in a similar fashion). It holds, for any $j\in K$
$$
\hat{F}_K(Y^j) = \frac{1}{|K|+1}\sum_{k\in K} \1_{Y^k \leq Y^j} = \frac{1}{|K|+1}\sum_{k\in K} \1_{X^k \leq X^j} = \frac{1}{|K|+1} \Big(1+\sum_{k\neq j,k\in K} \1_{X^k \leq X^j}\Big).
$$
By independence
\begin{equation*}
\begin{split}
\E \left|{X}^{\text{prox}}_K\right|^p 
= \int_{\mathbb{R}^{|K|}} \left|\Phi^{-1}\left(\frac{1}{|K|+1} \Big(1+\sum_{k=1,k\neq j}^{|K|} \1_{x_k \leq x_j}\Big)\right)\right|^p d \Phi(x_1)\ldots d\Phi(x_j)\ldots d\Phi(x_{|K|}).
\end{split}
\end{equation*}
By considering the cases where exactly $m=0,1,\ldots, |K|-1$ of the variables are less than $x_j$, integrating first with respect to variables $x_k,k\neq j$, and using change of variable $y = \Phi(x_j)$ we get
\begin{equation*}
\begin{split}
&\int_{\mathbb{R}^{|K|}} \left|\Phi^{-1}\left(\frac{1}{|K|+1} \Big(1+\sum_{k=1,k\neq j}^{|K|} \1_{x_k \leq x_j}\Big)\right)\right|^p d \Phi(x_1)\ldots d\Phi(x_j)\ldots d\Phi(x_|K|)\\
&= \sum_{m=0}^{|K|-1} \binom{|K|-1}{m}\int_{\mathbb{R}} \left|\Phi^{-1}\left(\frac{m+1}{|K|+1}\right)\right|^p \Phi(x_j)^m (1-\Phi(x_j))^{|K|-1-m} d\Phi(x_j) \\
&= \sum_{m=0}^{|K|-1} \binom{|K|-1}{m}\int_0^1 \left|\Phi^{-1}\left(\frac{m+1}{|K|+1}\right)\right|^p y^m (1-y)^{|K|-1-m} dy\\
&= \sum_{m=0}^{|K|-1} \binom{|K|-1}{m} \left|\Phi^{-1}\left(\frac{m+1}{|K|+1}\right)\right|^p B(m+1,|K|-m),
\end{split}
\end{equation*}
where $B(x,y)$ denotes the Beta function. Since
$$
\binom{|K|-1}{m}B(m+1,|K|-m) = \frac{(|K|-1)!}{m!(|K|-m-1)!}\cdot\frac{m!(|K|-m-1)!}{|K|!} = \frac{1}{|K|},
$$
we have
\begin{align*}
\sum_{m=0}^{|K|-1} &\binom{|K|-1}{m} \left|\Phi^{-1}\left(\frac{m+1}{|K|+1}\right)\right|^p B(m+1,|K|-m) = \frac{1}{|K|}\sum_{m=0}^{|K|-1} \left|\Phi^{-1}\left(\frac{m+1}{|K|+1}\right)\right|^p\\[2mm]
&= \frac{1}{|K|}\sum_{m\leq \frac{|K|-1}{2}} \left|\Phi^{-1}\left(\frac{m+1}{|K|+1}\right)\right|^p  + \frac{1}{|K|}\sum_{m\geq \frac{|K|-1}{2}} \left|\Phi^{-1}\left(\frac{m+1}{|K|+1}\right)\right|^p=(A)+(B).
\end{align*}
In (A), $m\leq \frac{|K|-1}{2}$ so that $\frac{m+1}{|K|+1}\leq \frac12$. 
Using the fact that $|\Phi^{-1}(x)|^p$ is decreasing on $x\in\left(0,\frac12\right]$ we get 
$$
\left|\Phi^{-1}\left(\frac{m+1}{|K|+1}\right)\right|^p \leq \left|\Phi^{-1}\left(\frac{x}{|K|+1}\right)\right|^p
$$
for all $x\in [m,m+1]$. Therefore (A) can be bounded as follows.
\begin{equation*}
\begin{split}
&(A) 
= \frac{1}{|K|}\sum_{m\leq \frac{|K|-1}{2}} \int_{m}^{m+1} \left|\Phi^{-1}\left(\frac{m+1}{|K|+1}\right)\right|^p dx\leq \frac{1}{|K|}\sum_{m\leq \frac{|K|-1}{2}} \int_{m}^{m+1} \left|\Phi^{-1}\left(\frac{x}{|K|+1}\right)\right|^p dx\\[2mm]
& =\frac{1}{|K|} \int_{0}^{\frac{|K|+1}{2}} \left|\Phi^{-1}\left(\frac{x}{|K|+1}\right)\right|^p dx\leq \frac{|K|+1}{|K|} \int_0^1 \left|\Phi^{-1}\left(y\right)\right|^p dy\leq 2 \int_{-\infty}^\infty |x|^p d \Phi(x) < \infty.
\end{split}
\end{equation*}
A similar argument holds for (B), where we use that $|\Phi^{-1}(x)|^p$ is increasing on $x\in\left[\frac12,1\right)$ and that, for all $x\in [m+1,m+2]$,
$$
\left|\Phi^{-1}\left(\frac{m+1}{|K|+1}\right)\right|^p \leq \left|\Phi^{-1}\left(\frac{x}{|K|+1}\right)\right|^p.
$$
Consequently, for every $p\geq 1$ the family $X^{\text{prox}}_K, |K|=1,2,\ldots$ is uniformly integrable. This concludes the proof.
\end{proof}
\begin{cor}
\label{cor:proxycov-convergence}
Let $t,s\in \pi_N$ be fixed and let $\hat{R}_J(t,s)$ be given by \eqref{eq:cov-proxy-estimation}. Then, for any $p\geq 1$ and as $|J|\to \infty$, we have $\hat{R}_J(t,s) \to R(t,s)$\textcolor{black}{, the covariance function of $X$,} in $L^p(\Omega)$. 
\end{cor}
\begin{proof}
The claim follows directly from Proposition \ref{pro:proxy-convergence} together with the fact that $|J|\leq\min(|K(t)|,|K(s)|)$ and that $X$ is Gaussian. 
\end{proof}
\begin{cor}
\label{cor:proxy-predictor-convergence}
Let $t\in \Delta\cap\pi_N$ be fixed and let the approximation $\hat{X}^j_{t,N}$ of the predictor be given by \eqref{eq:prediction-based-on-proxy}. Then, as $\min_{(t,s)\in \pi_N^{\textcolor{black}{2}}}|J(t,s)|\to \infty$,
$$
\hat{X}^j_{t,N} \to \hat{X}_{t,N}
$$
in probability, where $\hat{X}_{t,N}$ is given by \eqref{eq:discrete-predictor}.
\end{cor}
\begin{proof}
The claim follows directly from Proposition \ref{pro:proxy-convergence} and Corollary \ref{cor:proxycov-convergence}.
\end{proof}
\begin{rem}
Note that the result above does not necessarily hold for the stronger $L^p(\Omega)$ convergence. Indeed, while Proposition \ref{pro:proxy-convergence} ensures that ${X}_{t,K}^{\text{prox}}$ converges in $L^p(\Omega)$, this is not necessarily true for the precision matrix $\hat{\textbf{R}}_{N,\Delta}^{-1}$. 
\end{rem}
\begin{lmm}
\label{lma:rho-estimator}
Let $\hat{\rho}_{M}(t,t|\Delta\cap\pi_N)$ be given by \eqref{eq:rho-estimator} and let $Z_{N}$ be a centered Gaussian vector with covariance $\hat{\textbf{R}}_{N}$. Let $\hat{\rho}(t,t|\Delta\cap\pi_N)$ denote the conditional variance associated to $Z_{N}$ and let $\rho(t,t|\Delta\cap\pi_N)$ denote the conditional variance associated to $X$, given by \eqref{eq:discrete-predictor}.
Then, as $M\to \infty$, we have  
$$
\hat{\rho}_{M}(t,t|\Delta\cap\pi_N) \to \hat{\rho}(t,t|\Delta\cap\pi_N)
$$
in probability. Moreover, $\hat{\rho}(t,t|\Delta\cap\pi_N)$ is deterministic and satisfies, as $\min_{(t,s)\in\pi_N^{\textcolor{black}{2}}}|J(t,s)|\to \infty$,
$$
\hat{\rho}(t,t|\Delta\cap\pi_N) \to \rho(t,t|\Delta\cap\pi_N).
$$
\end{lmm}
\begin{proof}
The first claim is a simple application of the law of large numbers. 
For the second claim, it suffices to observe that $\hat{\rho}(t,t|\Delta\cap\pi_N)$ depends only on the covariance $\hat{\mathbf{R}}_N$ given by \eqref{eq:cov-proxy-estimation}. Thus the claim follows directly from Corollary \ref{cor:proxycov-convergence}. 
\end{proof}
\begin{lmm}
\label{lma:integral}
Let $L>0$ be fixed and let $\phi(z)$ be the density function of the standard normal distribution. Then for every bounded function $f$ the mapping $\textcolor{black}{\mathbb{R}\times (0,\infty)} \mapsto \mathbb{R}$ defined by 
$$
(x,c) \mapsto \int_{-L}^L f(x+zc)\phi(z)dz 
$$
is continuous.
\end{lmm}
\begin{proof}
Let $x_n\to x\textcolor{black}{\in \mathbb{R}}$ and $c_n\to c\textcolor{black}{>0}$. Performing a change of variable $z = \frac{x-x_n}{c_n} + \frac{cy}{c_n}$ we see that
$$
\int_{-L}^L f(x_n+zc_n)\phi(z)dz = \frac{c}{\sqrt{2\pi}c_n} \int_{A_n}^{B_n} f(x+cy)\exp\left(-\frac12\left(\frac{x-x_n}{c_n} + \frac{cy}{c_n}\right)^2\right)dy,
$$
where 
$$
A_n = \frac{x_n-x}{c} - \frac{c_nL}{c}\qquad\textrm{ and }\qquad
B_n = \frac{x_n-x}{c} + \frac{c_nL}{c}.
$$
Since $A_n \to -L$ and $B_n \to L$, the claim follows immediately by continuity of the exponential function and the fact that boundedness of $f$ allows us to apply dominated convergence theorem.
\end{proof}
\begin{rem}
\changes{
We actually apply the above lemma for continuous function $f$ and for $c\in(0,\infty)$. Note that in the case of continuous $f$, we have continuity over $(x,c)\in \mathbb{R}^2$.
}
\end{rem}
\begin{proof}[Proof of Theorem \ref{thm:data-driven-predictor}]
Let $t$ and $j$ be fixed. We proceed by studying the limits one by one and divide the proof accordingly into four steps.\\
\textbf{Step 1 (limit $M\to \infty$):} \\
Since the function $\hat{f}_K$ is bounded, Lemma \ref{lma:rho-estimator} and Lemma \ref{lma:integral} give
$$
\lim_{M\to \infty} \hat{Y}_{t,L,J^\ast,N,M} = \int_{-L}^L \hat{f}_K\left(t, \hat{X}^j_{t,N} + z \sqrt{\hat{\rho}(t,t|\Delta\cap\pi_N)}\right) \phi(z)dz.
$$
\textbf{Step 2 (limit $J^\ast \to \infty$):}\\
We write, using in the sequel the notation $\hat{X}_{t,N,j}$ for the discrete approximation of $\mathbb{E}[X^j_t|\mathcal{F}_{\Delta}]$ obtained via Theorem~\ref{thm:discrete},
\begin{equation*}
\begin{split}
&\int_{-L}^L \hat{f}_K\left(t, \hat{X}^j_{t,N} + z \sqrt{\hat{\rho}(t,t|\Delta\cap\pi_N)}\right) \phi(z)dz - \int_{-L}^L f\left(t, \hat{X}_{t,N,j} + z \sqrt{\rho(t,t|\Delta\cap\pi_N)}\right) \phi(z)dz\\
&= \int_{-L}^L \hat{f}_K\left(t, \hat{X}^j_{t,N} + z \sqrt{\hat{\rho}(t,t|\Delta\cap\pi_N)}\right) \phi(z)dz - \int_{-L}^L f\left(t, \hat{X}^j_{t,N} + z \sqrt{\hat{\rho}(t,t|\Delta\cap\pi_N)}\right) \phi(z)dz\\
&+\int_{-L}^L f\left(t, \hat{X}^j_{t,N} + z \sqrt{\hat{\rho}(t,t|\Delta\cap\pi_N)}\right) \phi(z)dz - \int_{-L}^L f\left(t, \hat{X}_{t,N,j} + z \sqrt{\rho(t,t|\Delta\cap\pi_N)}\right) \phi(z)dz.
\end{split}
\end{equation*}

For the second difference, note that $\hat{X}^j_{t,N} \to \hat{X}_{t,N,j}$ by Corollary \ref{cor:proxy-predictor-convergence} and $\hat{\rho}(t,t|\Delta\cap\pi_N) \to \rho(t,t|\Delta\cap\pi_N)$ by Lemma \ref{lma:rho-estimator}. 
Since $f$ is continuous, we may apply Lemma \ref{lma:integral} and the continuous mapping theorem to conclude
$$
\int_{-L}^L f\left(t, \hat{X}^j_{t,N} + z \sqrt{\hat{\rho}(t,t|\Delta\cap\pi_N)}\right) \phi(z)dz - \int_{-L}^L f\left(t, \hat{X}_{t,N,j} + z \sqrt{\rho(t,t|\Delta\cap\pi_N)}\right) \phi(z)dz \to 0.
$$
The first difference converges to $0$ as well. To see this, we apply the fact that the empirical quantile function $\hat{Q}_{t,K}$ converges uniformly (and almost surely) to $Q_t$ on compact intervals. Hence, $\hat{f}_K$ converges uniformly to $f$ as well. 

Thus, we have 
$$
\lim_{J^\ast\to \infty} \int_{-L}^L \hat{f}_K\left(t, \hat{X}^j_{t,N} + z \sqrt{\hat{\rho}(t,t|\Delta\cap\pi_N)}\right) \phi(z)dz = \int_{-L}^L f\left(t, \hat{X}_{t,N,j} + z \sqrt{\rho(t,t|\Delta\cap\pi_N)}\right) \phi(z)dz.
$$

\textbf{Step 3 (limit $N\to \infty$):}\\
By Theorem \ref{thm:discrete} we have $\hat{X}_{t,N,j} \to \hat X^j_t=\mathbb{E}[X^j_t|\mathcal{F}_\Delta]$ in $L^2(\Omega)$ as $N\to\infty$. 
Consequently, we also have $\rho(t,t|\Delta\cap\pi_N) \to \rho(t,t|\Delta)$. 
As in the previous step, we may apply continuity of $f$, Lemma \ref{lma:integral}, and the continuous mapping theorem to get
$$
\lim_{N\to \infty}\int_{-L}^L f\left(t, \hat{X}_{t,N,j} + z \sqrt{\rho(t,t|\Delta\cap\pi_N)}\right) \phi(z)dz = \int_{-L}^L f\left(t, \hat{X}^j_{t} + z \sqrt{\rho(t,t|\Delta)}\right) \phi(z)dz.
$$
\textbf{Step 4 (limit $L \to \infty$):}\\
Suppose first that $f\geq 0$. In this case the monotone convergence theorem applies and we obtain
$$
\lim_{L\to \infty}\int_{-L}^L f\left(t, \hat{X}^j_{t} + z \sqrt{\rho(t,t|\Delta)}\right) \phi(z)dz= \int_{-\infty}^\infty f\left(t, \hat{X}^j_{t} + z \sqrt{\rho(t,t|\Delta)}\right) \phi(z)dz.
$$
In the general case, bijectivity of $x\mapsto f(t,x)$ implies that $f$ changes sign only once. Considering integrals separately before and after this change point and applying the monotone convergence theorem concludes step 4 and the whole proof.

\end{proof}  
\textcolor{black}{
\begin{proof}[Proof of Proposition \ref{prop:convergence-rate}]
We have 
\begin{equation*}
\small{\begin{split}
&\hat{Y}_{t,L,M,K,N}-\left[Y_t | \mathcal{F}_{\Delta}^Y\right] \\
&= \int_{-L}^L \hat{f}_K\left(t, \hat{X}^j_{t,N} + z \sqrt{\hat{\rho}_M(t,t|\Delta\cap\pi_N)}\right) \phi(z)dz -\int_{-\infty}^\infty f\left(t, \hat{X}^j_{t} + z \sqrt{\rho(t,t|\Delta)}\right) \phi(z)dz\\
&=\int_{-L}^L \hat{f}_K\left(t, \hat{X}^j_{t,N} + z \sqrt{\hat{\rho}_M(t,t|\Delta\cap\pi_N)}\right) \phi(z)dz - \int_{-L}^L f\left(t, \hat{X}^j_{t,N} + z \sqrt{\hat{\rho}_M(t,t|\Delta\cap\pi_N)}\right) \phi(z)dz\\
&+\int_{-L}^L f\left(t, \hat{X}^j_{t,N} + z \sqrt{\hat{\rho}_M(t,t|\Delta\cap\pi_N)}\right) \phi(z)dz - \int_{-L}^L f\left(t, \hat{X}_{t,N,j} + z \sqrt{\rho(t,t|\Delta\cap\pi_N)}\right) \phi(z)dz\\
&+\int_{-L}^L f\left(t, \hat{X}_{t,N,j} + z \sqrt{\rho(t,t|\Delta\cap\pi_N)}\right) \phi(z)dz  - \int_{-L}^L f\left(t, \hat{X}^j_{t} + z \sqrt{\rho(t,t|\Delta)}\right) \phi(z)dz\\
&+ \int_{-L}^L f\left(t, \hat{X}^j_{t} + z \sqrt{\rho(t,t|\Delta)}\right) \phi(z)dz - \int_{-\infty}^\infty f\left(t, \hat{X}^j_{t} + z \sqrt{\rho(t,t|\Delta)}\right) \phi(z)dz\\
&=: I_1 + I_2 + I_3 + I_4.
\end{split}
}
\end{equation*}
Hence it suffices to bound terms $I_1$ to $I_4$, that we now consider one by one. \\[2mm]
\textbf{Term $I_1$:} 
Using bijectivity and smoothness of $f$, it follows from Proposition \ref{prop:f-rep} that $F_t$ is continuously differentiable with non-vanishing derivative on the regions where the quantile function $Q_t$ is non-constant. 
Hence, following \cite{vanderVaart} (see example 3.9.24 page 387), the empirical quantile function $\hat{f}_K$ converges to $f$ uniformly on a compact set $U$ with rate $O_p\left(\frac{1}{\sqrt{K}}\right)$. 
Consequently, we obtain $I_1 = O_p\left(\frac{C_U}{\sqrt{K}}\right)$. 
Note that, in $I_1$, we consider 
$$
U=\left\{
\hat{X}^j_{t,N} + z\sqrt{\hat{\rho}_M(t,t|\Delta \cap \pi_N)}\ \middle|\ z\in[-L,L]
\right\}.
$$
Since $\hat{X}^j_{t,N} \to \hat{X}_{t,N,j}$ as $K\to \infty$ and $\hat{\rho}_M(t,t|\Delta \cap \pi_N)\to \rho(t,t|\Delta \cap \pi_N)$ as $M\to \infty$, $C_U$ is a random constant depending solely on $N$ and $L$.  Hence, 
$$
I_1 = O_p\left(\frac{C_{L,N}}{\sqrt{K}}\right).
$$
\textbf{Term $I_2$}: 
For the term $I_2$, we use Lipschitz continuity of $f$ and differentiability of the square root to get
\begin{equation*}
\small{\begin{split}
&\left|\int_{-L}^L f\left(t, \hat{X}^j_{t,N} + z \sqrt{\hat{\rho}_M(t,t|\Delta\cap\pi_N)}\right) \phi(z)dz - \int_{-L}^L f\left(t, \hat{X}_{t,N,j} + z \sqrt{\rho(t,t|\Delta\cap\pi_N)}\right) \phi(z)dz\right|\\[1mm]
&\leq C \left(|\hat{X}^j_{t,N} - \hat{X}_{t,N,j}| + |\hat{\rho}_M(t,t|\Delta\cap\pi_N) - \rho(t,t|\Delta\cap\pi_N)|\right).
\end{split}
}
\end{equation*}
Here, we note that $f$ is continuously differentiable, hence Lipschitz continuous on compact sets. In this case, the constant $C$ relates to the set 
\small{$$
U=\left\{
\hat{X}^j_{t,N} + z\sqrt{\hat{\rho}_M(t,t|\Delta \cap \pi_N)}\ \middle|\ z\in[-L,L]
\right\}\cup \left\{\hat{X}_{t,N,j} + z\sqrt{\rho(t,t|\Delta \cap \pi_N)}\ \middle|\ z\in[-L,L]
\right\}.
$$
}
By the same arguments as above, hence we can take the constant $C$ to depend solely on $L$ and $N$. 
Therefore, in order to bound $I_2$, it suffices to bound 
$$
\left(|\hat{X}^j_{t,N} - \hat{X}_{t,N,j}| + |\hat{\rho}_M(t,t|\Delta\cap\pi_N) - \rho(t,t|\Delta\cap\pi_N)|\right).
$$
Recall that here 
\begin{equation}
\label{eq:prediction-recall}
\hat{X}_{t,N,j} = \sum_{t_j \in \Delta_N} a_j(t) X_{t_j} 
\end{equation}
with $\textbf{a}(t) = \textbf{R}_{N}^{-1}\textbf{b}_{N}(t)$, 
and $\hat{X}^j_{t,N}$ given similarly as
\begin{equation}
\label{eq:prediction-approx-recall}
\hat{X}_{t,N}^j = (\hat{a}_j(t))'X_{\Delta_j,K}^{\text{prox}},
\end{equation}
where $
\hat{a}_j(t) = \hat{\textbf{R}}_{N,\Delta_j}^{-1}\hat{\textbf{b}}_{N,\Delta_j}(t)
$
is computed through the estimated covariance (based on proxies). Recall next that 
$$
{X}^{\text{prox}}_{t,K} =\Phi_t^{-1}\left(\hat{F}_{t,K}(Y_t)\right)
$$
and 
$$
X_t = \Phi_t^{-1}\left(F_t(Y_t)\right).
$$
Thus, we obtain by the standard Delta method and Donsker's Theorem that, for any fixed $t \in \pi_N$, we have 
$$
{X}^{\text{prox}}_{t,K} - X_t = O_p\left(\frac{1}{\sqrt{K}}\right).
$$
From this, it readily follows that, for any $(t,s) \in \pi_N^2$, we have
$$
\hat{R}_J(t,s) - R(t,s) = O_p\left(\frac{1}{\sqrt{K}}\right).
$$
Using that, for any invertible matrices $A$ and $B$, 
$
(A^{-1} - B^{-1}) = A^{-1}(B-A)B^{-1}
$
leads to 
$$
\hat{\textbf{R}}_{N,\Delta_j}^{-1} - \textbf{R}_{N}^{-1} = O_p\left(\hat{\textbf{R}}_{N,\Delta_j} - \textbf{R}_{N}\right).
$$
This, together with \eqref{eq:prediction-recall} and \eqref{eq:prediction-approx-recall} gives
$$
\hat{X}^j_{t,N} - \hat{X}_{t,N,j} = O_p\left(\frac{C_N}{\sqrt{K}}\right).
$$
For the difference  $
\hat{\rho}_M(t,t|\Delta\cap\pi_N) - \rho(t,t|\Delta\cap\pi_N),
$
we write 
$$
\hat{\rho}_M(t,t|\Delta\cap\pi_N) - \rho(t,t|\Delta\cap\pi_N) =\hat{\rho}_M(t,t|\Delta\cap\pi_N) - \hat{\rho}(t,t|\Delta\cap\pi_N) 
+\hat{\rho}(t,t|\Delta\cap\pi_N)- \rho(t,t|\Delta\cap\pi_N),
$$
where $\hat{\rho}(t,t|\Delta\cap\pi_N)$ is the conditional variance associated to a centered Gaussian vector $Z_N$ with covariance $\hat{\textbf{R}}_{N}$. Here, 
$$
\hat{\rho}_M(t,t|\Delta\cap\pi_N) - \hat{\rho}(t,t|\Delta\cap\pi_N) = O_p\left(\frac{C_N}{\sqrt{M}}\right)
$$
as $\hat{\rho}_M$ is computed using a standard Monte-Carlo method. For the remaining term 
$
\hat{\rho}(t,t|\Delta\cap\pi_N)- \rho(t,t|\Delta\cap\pi_N), 
$
it suffices to observe that $\hat{\rho}(t,t|\Delta\cap\pi_N)$ can be computed  by using the covariances $\hat{R}_N(t,s),\ (t,s)\in \pi^2_N$ while $\rho(t,t|\Delta\cap\pi_N)$ can be computed by using the true covariances $R(t,s), \ (t,s)\in \pi^2_N$. Hence, by similar considerations as above, we deduce 
$$
\hat{\rho}(t,t|\Delta\cap\pi_N)- \rho(t,t|\Delta\cap\pi_N) = O_p\left(\frac{C_N}{\sqrt{K}}\right).
$$
Combining these estimates gives 
$$
I_2 = O_p\left(\frac{C_{L,N}}{\sqrt{M}}\right) +O_p\left(\frac{C_{L,N}}{\sqrt{K}}\right).
$$
\textbf{Term $I_3$:} As for $I_2$, we can use Lipschitz continuity to get 
\begin{equation*}
\begin{split}
&\left|\int_{-L}^L f\left(t, \hat{X}_{t,N,j} + z \sqrt{\rho(t,t|\Delta\cap\pi_N)}\right) \phi(z)dz  - \int_{-L}^L f\left(t, \hat{X}^j_{t} + z \sqrt{\rho(t,t|\Delta)}\right) \phi(z)dz\right|\\[1mm]
&\leq C \left(|\hat{X}_{t,N,j}- \hat{X}^j_{t}| + |\sqrt{\rho(t,t|\Delta\cap \pi_N)} - \sqrt{\rho(t,t|\Delta)}|\right)
\end{split}
\end{equation*}
where, in this case, the constant $C$ is determined by the set 
\small{$$
U=\left\{\hat{X}_{t,N,j} + z\sqrt{\rho(t,t|\Delta \cap \pi_N)}\ \middle|\ z\in[-L,L]
\right\}\cup \left\{
\hat{X}^j_{t} + z\sqrt{\rho(t,t|\Delta)}\ \middle|\ z\in[-L,L]
\right\}.
$$
}
Since now $\hat{X}_{t,N,j} \to \hat{X}^j_t$ as $N\to \infty$, the constant $C$ can be taken to depend only on the parameter $L$. The rate in $N$ for $\hat{X}_{t,N,j}- \hat{X}^j_{t}$ is determined by the rate for which $\mathcal{F}_{\Delta_N}$ increases to $\mathcal{F}_{\Delta}$ as 
$$
\hat{X}_{t,N,j}- \hat{X}^j_{t} = \E\left[X_t | \mathcal{F}_{\Delta_N}\right] - \E\left[X_t | \mathcal{F}_{\Delta}\right]
$$
and, by Markov inequality, 
$$
|\hat{X}_{t,N,j}- \hat{X}^j_{t}| = O_p\left(\sqrt{\E\left(\hat{X}_{t,N,j}- \hat{X}^j_{t}\right)^2}\right).
$$
For the term $\sqrt{\rho(t,t|\Delta\cap \pi_N)} - \sqrt{\rho(t,t|\Delta)}$ we have, by using the triangle inequality for the $L^2(\Omega)$-norm,
$$
|\sqrt{\rho(t,t|\Delta\cap \pi_N)} - \sqrt{\rho(t,t|\Delta)}|= \left|\sqrt{\E(X_t^j - \hat{X}_{t,N,j})^2} - \sqrt{\E(X^j_t - \hat{X}^j_t)^2}\right|\leq \sqrt{\E(\hat{X}_t^j - \hat{X}_{t,N,j})^2}.
$$
Thus, 
$$
I_3= O_p\left(C_{L}\sqrt{\E\left(\hat{X}_{t,N,j}- \hat{X}^j_{t}\right)^2}\right).
$$
Noting that $I_4$ is exactly the final term appearing in Proposition \ref{prop:convergence-rate}, this completes the whole proof.
\end{proof}
}






\bibliographystyle{plain}
\bibliography{Bibliography}           
\vspace{3mm} 


\end{document}